\documentclass[12pt,a4paper]{amsart}
\usepackage{amssymb}
\usepackage{amsmath}
\usepackage{xcolor}
\usepackage{amscd}
\usepackage{pstricks}
\usepackage[makeroom]{cancel}
\usepackage{pst-node}

\usepackage{geometry}
\newtheorem{thm}{Theorem}[section]
\newtheorem{cor}[thm]{Corollary}
\newtheorem{lem}[thm]{Lemma}

\theoremstyle{definition}

\newtheorem{defn}[thm]{Definition}

\usepackage{graphics}
\usepackage{graphicx}

\newtheorem{rem}[thm]{Remark}

\numberwithin{equation}{section}

\begin{document}

\title[Subelliptic heat kernel on $\mathbb{S}^7$]{Trivializable and quaternionic subriemannian structure on $\mathbb{S}^7$ and subelliptic heat kernel}

\author{W. Bauer and A. Laaroussi}
 \thanks{Both authors have been supported by the priority program SPP 2026 {\it geometry at infinity} of Deutsche Forschungsgemeinschaft (project number BA 3793/6-1).}


\address{Wolfram Bauer\endgraf
Institut f\"{u}r Analysis, Leibniz Universit\"{a}t  \endgraf
Welfengarten 1, 30167 Hannover, Germany \endgraf
}
\email{bauer@math.uni-hannover.de}

\address{Abdellah Laaroussi \endgraf
 Institut f\"{u}r Analysis, Leibniz Universit\"{a}t \endgraf
Welfengarten 1, 30167 Hannover, Germany\endgraf
}
\email{abdellah.laaroussi@math.uni-hannover.de}

\begin{abstract}
 On the seven dimensional Euclidean sphere $\mathbb{S}^7$ we compare two subriemannian structures with regards to various geometric 
and analytical properties. The first structure is called trivializable and the underlying distribution $\mathcal{H}_T$ is induced by a Clifford module 
structure of $\mathbb{R}^8$. More precisely, $\mathcal{H}_T$ is rank $4$,  bracket generating of step two and generated by globally defined vector fields. 
The distribution $\mathcal{H}_{Q}$ of the second structure  is of rank 4 and step two as well and obtained as the horizontal distribution in the quaternionic Hopf fibration $\mathbb{S}^3\hookrightarrow\mathbb{S}^7\rightarrow\mathbb{S}^4$. Answering a question in \cite{MaMo11} we first show that $\mathcal{H}_{Q}$ does not admit a 
global nowhere vanishing smooth section. In both cases we determine the Popp measures, the  intrinsic sublaplacians $\Delta_{\textup{sub}}^T$ and 
$\Delta_{\textup{sub}}^{Q}$ and the nilpotent approximations. We conclude that both subriemannian structures are not locally isometric and we discuss properties of the isometry group. By determining the first heat invariant of the sublaplacians it is shown that both structures are also not isospectral in the subriemannian sense.
\end{abstract}
\maketitle
\tableofcontents
\thispagestyle{empty}
\section{Introduction}
Let $(M,\mathcal{H},\langle\cdot,\cdot\rangle)$ be a subriemannian manifold, i.e. $M$ is a smooth connected orientable manifold endowed with a bracket generating subbundle $\mathcal{H}$ of the tangent bundle $TM$. Moreover,  $\langle\cdot,\cdot\rangle$ denotes a family of inner products on $\mathcal{H}$ which smoothly vary with the base point.  From a geometric point of view one is led to the problem of defining and classifying subriemannian structures of specific types on a given manifold 
(e.g. up to local subriemannian isometries) or, to compare them with regards to various of their geometric properties, e.g. \cite{BGMR,BFI0,MaMo11,Mo}. Any regular subriemannian structure on $M$ induces a hypoelliptic sublaplacian $\Delta_{\textup{sub}}$ which intrinsically is defined based on the Popp measure construction, \cite{ABB,ABGF,Bariz}. From an analytical point of view one may study the diffusion on $M$ generated by the heat operator induced by $\Delta_{\textup{sub}}$. Which geometric data can be recovered from such analytically defined objects? Extending a classical problem in Riemannian geometry and in the case of a compact manifold $M$, one may ask whether two non-isometric subriemannian structures are isospectral with respect to their induced sublaplacians (e.g. see \cite{BFIL} for positive examples). 
\vspace{1mm}\par
In the special case of a Euclidean 
sphere $M= \mathbb{S}^N$ of dimension $N$ typical methods (depending on $N$) of installing a subriemannian geometry on $M$ use a Lie group structure ($N=3$), a contact structure ($N$ odd), a principle bundle structure such as the Hopf fibration $\mathbb{S}^{2n+1} \rightarrow \mathbb{CP}^n$ ($N$ odd) or the quaterionic Hopf fibration, $\mathbb{S}^{4n+3} \rightarrow \mathbb{HP}^n$, CR-geometry or a suitable number of canonical vector fields in \cite{A} ($N=3,7,15$). In the lowest dimensional case $N=3$ all these structures essentially coincide as was pointed out in \cite{MaMo11}. 
\vspace{1mm}\par 
The present paper compares two of the above mentioned subriemannian structures on $M= \mathbb{S}^7$ and, in particular, it extends results in \cite{BFI}. 
Therein the authors have shown that the $N$-dimensional Euclidean sphere $\mathbb{S}^N$ carries a trivializable subriemannian structure induced by a Clifford module structure of $\mathbb{R}^{N+1}$ only in dimensions $N=3,7,15$. Moreover, in this paper the spectum of a corresponding second order differential operator (in \cite{BFI} it is called sublaplacian) has been studied. However, it should be pointed out that this sublaplacian differs from the intrinsic one which we consider here by a first order term. 
\vspace{1ex}\par 
We recall the construction of a bracket generating trivial rank-$k$ distribution on a sphere of dimension $N=3,7,15$: Consider a family of $(N+1)\times(N+1)$ skew-symmetric real matrices $A_1,\cdots,A_k$ such that
$$A_iA_j+A_jA_i=-2\delta_{ij} \hspace{3ex} \mbox{\it for} \hspace{3ex} i,j=1,\cdots,k.$$
Then a collection of $k$ linear vector fields on $\mathbb{S}^N$ that are orthonormal at each point of the sphere can be defined  in global coordinates of 
$\mathbb{R}^{N+1}$ by:
 $$X(A_l):=\sum_{i,j=1}^{N+1}(A_l)_{ij}x_j\frac{\partial}{\partial x_i}, \hspace{4ex} (l=1,\cdots,k).$$
According to the results in \cite{BFI} the rank $k$ distribution 
$$\mathcal{H}:=\text{Span}\{X(A_l):l=1,\cdots,k\} \subset T\mathbb{S}^N$$ 
is of step two and trivial as a vector bundle by definition.  Moreover, $\mathcal{H}$ is bracket generating only for particular choices of $N$ and $k$. 
Whereas in the case $N=3$ and $N=15$ a trivializable bracket generating distribution $\mathcal{H}\subsetneq T\mathbb{S}^N$ of the above kind must have rank two and rank eight, respectively, there are three trivializable subriemannian structures on $\mathbb{S}^7$ of rank $4,5$ and $6$. For $N=3$ such a subriemannian structure on $\mathbb{S}^3$ is isometric to the one induced by the well-known Hof fibration, see \cite{MaMo11}.  
 $$\mathbb{S}^1\hookrightarrow\mathbb{S}^3\longrightarrow\mathbb{S}^2.$$
Under some geometric aspects the above trivializable SR structures on $\mathbb{S}^7$ have been studied in \cite{BT}. More precisely, the authors analyzed the corresponding geodesic flow and constructed a family of normal subriemannian geodesics  (i.e. locally length minimizing curves induced from the geodesic equations).\ \par
 In the present paper, we analyze a trivializable subriemannian structure on $\mathbb{S}^7$ of rank $4$ and we compare it with the quaternionic contact  structure of rank $4$ on $\mathbb{S}^7$ induced by the quaternionic Hopf fibration \cite{MaMo12}
 $$\mathbb{S}^3\hookrightarrow\mathbb{S}^7\rightarrow\mathbb{S}^4.$$
 Answering a question in \cite{MaMo11} we first show that the horizontal distribution in the quaternionic Hopf fibration is not trivial. It does not even admit a 
single global nowhere vanishing smooth section. In fact, this follows from results in topological K-theory in \cite{Mah}.  We show that the so-called  tangent groups i.e. local approximations of the trivializable subriemannian structure on $\mathbb{S}^7$  may change from point to point. As a consequence the subriemannian isometry group cannot act transitively on $\mathbb{S}^7$. Furthermore, the trivializable distribution is of elliptic type  (see \cite{Mo} for a definition) inside an open dense subset. Hence, by a result of R. Montgomery in \cite{Mo}, it follows that the subriemannian isometry group is finite dimensional with dimension bounded by $21$. 
\vspace{1mm}
\par 
 We calculate the Popp measures on 
$\mathbb{S}^7$ induced by the trivializable and quaternionic contact structures, respectively, and we determine the intrinsic sublaplacians.  
Moreover,  by applying recent results due to Y.C. de Verdi\'{e}re, L. Hillairet and E. Tr\'{e}lat in \cite{Verd} combined with an explicit form of the subelliptic heat kernel on step two nilpotent Lie groups in \cite{BGG,CCFI}  we compute the first heat invariants appearing in the small-time asymptotics of the heat trace associated to the intrinsic sublaplacians.  Based on these data we can show that the subriemannian structures (quaternionic contact and trivializable) on $\mathbb{S}^7$ are neither locally isometric nor isospectral with respect to the intrinsic sublaplacians. 
\vspace{1mm}
\par 
 Finally, we mention that an explicit form of the heat kernel (i.e. fundamental solution to the heat operator) of the intrinsic sublaplacian 
induced from the trivializable subriemannian structure is unknown. Since the corresponding subriemannian isometry group does not act transitively on $\mathbb{S}^7$ it would be not sufficient to only calculated it at a fixed point. This is in contrast to the quaternionic contact structure. In the latter case the isometry  group acts transitively and the 
subelliptic heat kernel has been obtained explicitly in \cite{Baud_Wang}. Moreover, the explicit form of the heat kernel has been used in \cite{Baud_Wang} to obtain 
some of the heat invariants, i.e. in this case the analysis does not rely on the approximation methods in Section \ref{small}. 
\vspace{1ex}\par 
 
The paper is organized as follows:  Section \ref{subgeo} provides basic concepts and definitions in subriemannian geometry. In Sections \ref{quaho} and \ref{trist} we recall the construction of two different subriemannian structures on $\mathbb{S}^7$ and we list some of their properties. Then we compute the Popp volume induced by these structures in Section \ref{pop}. In Section \ref{nilapp} we show that the tangent groups of $\mathbb{S}^7$ endowed with the trivializable subriemannian structure 
 may change from point to point and that this structure is not locally isometric to the quaternionic contact structure.   The type of the trivializable structure is determined in Section \ref{type} and this allows us to obtain a bound on the dimension of the isometry group. In Section \ref{small} we compute the first heat invariants in the small-time asymptotics of the heat trace by using an approximation method in \cite{Chpo,Verd}. Comparing both we show that the above subriemannian structures on 
$\mathbb{S}^7$ are not isospectral with respect to the sublaplacians.  In Chapter \ref{Chapter_9} we consider the (non-intrinsic) sublaplacian $\widetilde{\Delta}_{\textup{sub}}^T$ on $\mathbb{S}^7_T$ induced by the standard measure on $\mathbb{S}^7$. In Theorem \ref{Lemma_inclusion_of_spectra_sublaplace} we prove the inclusion $ \sigma(\Delta^{\textup{Q}}_{\textup{sub}}) \subset \sigma (\widetilde{\Delta}_{\textup{sub}}^{\textup{T}})$ of spectra where $\Delta^{\textup{Q}}_{\textup{sub}}$ denotes the sublaplacian corresponding to the quarternionic contact structure. However, we mention that both operators are not isospectral. Chapter \ref{Chapter_9}   extends former results in \cite{BFI}. 

 \section{Subriemannian geometry}\label{subgeo}
 We start recalling basic definitions in subriemannian geometry  \cite{ABB, Mo,S1,S2}. 
 \vspace{1mm}\par 
 A subriemannian manifold  (shortly: SR manifold) is a triple $(M,\mathcal{H},\langle \cdot,\cdot\rangle)$ where
 \begin{itemize}
 \item[(a)] $M$ is a connected smooth manifold of dimension $n$.
 \item[(b)] $\mathcal{H}$ is a smooth distribution of constant rank $k<n$  which we may identify with the sheaf of smooth vector fields tangent to $\mathcal{H}$ 
 (horizontal vector fields). We  assume that $\mathcal{H}$ is bracket generating, i.e. if we set for $j\geq 1$
 $$\mathcal{H}^{1} :=\mathcal{H}\text{ and } \mathcal{H}^{j+1}:=\mathcal{H}^{j}+[\mathcal{H},\mathcal{H}^{j}],$$
 then  for each $q \in M$  there is $p\in\mathbb{N}$ such that $\mathcal{H}_q^p=T_qM.$
 \item[(c)] $\langle \cdot,\cdot\rangle$ is a fiber inner product on $\mathcal{H}$, i.e.
 $$\langle \cdot,\cdot\rangle_q:\mathcal{H}_q\times\mathcal{H}_q\longrightarrow\mathbb{R}$$
  is an inner product for all $q\in M$ and it smoothly varies with $q\in M$.
 \end{itemize}

We call a subriemannian manifold $(M,\mathcal{H},\langle \cdot,\cdot\rangle)$  {\it regular}, if for all $j\geq 1$ the dimension of $\mathcal{H}^{j}_q$ does not depend on the point $q\in M$. Furthermore, a  regular SR manifold $M$ is said to be of step $r$ if $r$ is the smallest integer such that $\mathcal{H}^r=TM$.\ \\
In this work we only consider regular subriemannian manifolds of step $2$. Therefore we recall the required concepts only in this case.\ \\

A local  frame $\{X_1,\cdots,X_m,X_{m+1},\cdots,X_n\}$ is called {\it adapted}, if  the vector fields 
$X_1,\cdots,X_m$ form a local orthonormal frame of 
 $(\mathcal{H}, \langle \cdot, \cdot\rangle )$.
 \par 
Given two subriemannian manifolds $(M,\mathcal{H},\langle \cdot,\cdot\rangle)$ and $(M^\prime,\mathcal{H}^\prime,\langle \cdot,\cdot\rangle^\prime)$, we call a map $\phi:M\rightarrow M^\prime$ {\it horizontal} if   its differential maps $\mathcal{H}$ to $\mathcal{H}^\prime$, i.e. $\phi_\ast(\mathcal{H})\subseteq\mathcal{H}^\prime$. 
\begin{defn}  $\phi$ is called (local) {\it subriemannian isometry} if it is a horizontal (local) diffeomorphism such that 
$\phi_\ast:(\mathcal{H},\langle \cdot,\cdot\rangle)\rightarrow(\mathcal{H}^\prime,\langle \cdot,\cdot\rangle^\prime)$ becomes an isometry.
\end{defn}

To every  regular subriemannian manifold $M$ of step $2$, a  family of graded $2$-step nilpotent Lie algebras
$$\mathfrak{g}M:=\mathcal{H}\oplus\left(\mathcal{H}^2\slash\mathcal{H}\right)$$
is associated with Lie brackets induced by the Lie brackets of vector fields on $M$,  \cite{Mo}. Note that the defined Lie brackets respect the above grading, i.e.
$$[\mathcal{H},\mathcal{H}]\subseteq \mathcal{H}^2\slash\mathcal{H}\hspace{3ex} \text{\it  and }\hspace{3ex} [\mathcal{H},\mathcal{H}^2\slash\mathcal{H}]=[\mathcal{H}^2\slash\mathcal{H},\mathcal{H}^2\slash\mathcal{H}]=0.$$
Hence, $\mathfrak{g}M$ is a  smooth family
of Carnot Lie algebras. We call  $\mathfrak{g}M(q)=\mathfrak{g}M_q$ the {\it nilpotent approximation} of $(M, \mathcal{H}, \langle \cdot, \cdot \rangle)$ at  $q \in M$, \cite{Mo}.
\vspace{1mm}
\par 
For every  $q\in M$, the {\it tangent group} $\mathbb{G}M(q)$ of $M$ at $q$ is the unique connected, simply connected nilpotent Lie group corresponding to the Lie algebra $\mathfrak{g}M(q)$. Note that $\mathbb{G}M$ is 
 not assumed to be locally trivial. In particular, the Lie groups $\mathbb{G}M(q)$ might  be non-isomorphic at different base points $q\in M$. \\

On a $7$-dimensional manifold there is a particular class of distributions called {\it elliptic}. Such distributions are interesting from a geometric point of view because the induced geometry has always a finite dimensional symmetry group. In the following we briefly recall how they are defined (see \cite{Mo} for more details). 
Let $\mathcal{H}$ be a co-rank $3$, bracket generating distribution of step two on a $7$-dimensional manifold $M$ and let us consider the so-called {\it curvature} (linear) bundle map of $\mathcal{H}$
\begin{equation}\label{curvature_map}
F:\Lambda^2\mathcal{H}\longrightarrow TM\slash\mathcal{H}
\end{equation}
defined by $F(X,Y)=-[X,Y]$ mod $\mathcal{H}$ for $X,Y\in\mathcal{H}$. Write $\mathcal{H}^\perp\subset T^\ast M$ for the bundle of covectors that annihilate 
$\mathcal{H}$. We consider now the dual curvature map $\omega$:
\begin{equation}\label{dual_curvature_map}
\omega:=F^\ast:\mathcal{H}^\perp\longrightarrow\Lambda^2\mathcal{H}^\ast.
\end{equation}
Since $\mathcal{H}$ is bracket generating, the curvature map is onto. Furthermore, the real vector space $\Lambda^4\mathcal{H}^\ast$ is $1$-dimensional, hence the squared dual curvature map
\begin{align*}
\omega^2&:\mathcal{H}^\perp\longrightarrow\Lambda^4\mathcal{H}^\ast\\
\lambda&\longmapsto \omega(\lambda)\wedge\omega(\lambda)
\end{align*}
is a quadratic form on the $3$-dimensional space $\mathcal{H}^\perp$ with values in the $1$-dimensional vector space $\Lambda^4\mathcal{H}^\ast$. We say that $\mathcal{H}$ is {\it elliptic} if this quadratic form has signature $(3,0)$ or $(0,3)$. Note that we do not have a canonical choice of an element in $\Lambda^4\mathcal{H}^\ast$ and hence, the signature is only defined up to a sign $\pm$. In general, we say that $\mathcal{H}$ is of type $(r,s)$  if this quadratic form has signature $(r,s)$ or $(s,r)$.

If $\mathcal{H}$ is of elliptic type then it was proven in \cite{Mo} that the symmetry group is always finite-dimensional and the maximal dimension of such group is realized by $\mathbb{S}^7$  endowed with the quaternionic contact structure.\ \\

 On a subriemannian manifold $M$ (not necessarily regular) the definition of a sublaplacian requires the  choice of a smooth measure $\mu$ on $M$, \cite{ABGF, Bariz,Verd}. 
We denote by $\text{div}_\mu$ the divergence operator associated with the measure $\mu$ defined by
$$\mathcal{L}_X\mu=\text{ div}_\mu(X)\mu$$
for every smooth vector field $X$ on $M$. Then we can associate to $\mu$ a {\it sublaplacian} $\Delta_{\textup{sub}}^\mu$ defined as the hypoelliptic, second 
order differential operator \cite{ABGF,Hoe67}: 
$$\Delta_{\textup{sub}}^\mu f:=-\text{div}_\mu\left(\nabla_\mathcal{H} f\right)\text{ for }f\in C^\infty(M).$$
Here $\nabla_\mathcal{H}$ denotes the horizontal gradient with respect to the horizontal metric $\langle\cdot,\cdot\rangle$ on $\mathcal{H}$,  which is defined at  
$q \in M$ by the properties: 
\begin{equation*}
\nabla_{\mathcal{H}}(\varphi)\in \mathcal{H}_q \hspace{2ex}\mbox{\it and} \hspace{2ex}  \langle \nabla_{\mathcal{H}}(\varphi), v\rangle_q= d\varphi(v), \hspace{3ex} v\in \mathcal{H}_q, \hspace{1ex} 
\varphi \in C^{\infty}(M). 
\end{equation*}
 Since the subriemannian manifold $M$ is assumed to be regular, there is a canonical choice of smooth measure on $M$ called {\it Popp measure} $\mu= \mathcal{P}$. 
 The sublaplacian $\Delta_{\textup{sub}}^{\mathcal{P}}$ defined from the Popp measure then is called the {\it intrinsic sublaplacian}  \cite{ABGF,Bariz,Mo}.
 \vspace{1mm}\par 
Note that the sublaplacian is positive and if  the manifold $M$ endowed with the subriemannian distance is complete, then  $\Delta_{\textup{sub}}^{\mu}$ is 
essentially selfadjoint  on $C_0^{\infty}(M)$ with unique selfadjoint extension on $L^2(M,\mu)$ (see  \cite{S1,S2,Verd}). Therefore the heat semigroup 
 $$\left(e^{-\frac{t}{2}\Delta_{\textup{sub}}^\mu}\right)_{t>0}$$ 
 is a well-defined one-parameter family of bounded operators on $L^2(M,\mu)$. In the following, we denote by ${K_t(\cdot,\cdot)}$ the heat kernel of the operator 
 $e^{-\frac{t}{2}\Delta_{\textup{sub}}^\mu}$ which is smooth due to the hypoellipticity of $\Delta_{\textup{sub}}^\mu$, \cite{Hoe67}.\ \\
 
 We recall the following formula for the small-time asymptotic expansion of the heat kernel on the diagonal, \cite{Bena, Verd}: for all $N\in\mathbb{N}$ and $q\in M$,
 $${K_t(q,q)}=\frac{1}{t^{Q(q)/2}}\left(c_0(q)+c_1(q)t+\cdots+c_N(q)t^N+o(t^N) \right)\hspace{3ex} \text{ \it as }t\to 0,$$
 which also holds in a non-regular situation and for an  arbitrary smooth measure $\mu$ in the definition of the sublaplacian. 
 Moreover, under the assumption that the subriemannian manifold is regular, the functions $c_i$ are smooth in a neighbourhood of $q$. Here $Q$ denotes the Hausdorff dimension of the metric space $(M,d)$ where $d$ is the subriemannian distance ({\it Carnot-Carath\'{e}odory distance}) on $M$, see \cite{Mo}. 
\section{Quaternionic Hopf structure}\label{quaho}
Let $\mathbb{H}\simeq \mathbb{R}^4$ denote the quaternionic space 
$$\mathbb{H}:=\{x+y{\bf i}+z{\bf j}+\omega {\bf k}:x,y,z,\omega\in\mathbb{R}\},$$
where ${\bf i}^2={\bf j}^2={\bf k^2}=-1$ and ${\bf ij}=-{\bf ji}={\bf k}$, ${\bf jk}=-{\bf kj}={\bf i}$ and ${\bf ki}=-{\bf ik}={\bf j}$.\ \\
For $n\geq 0$, we consider the $(n+1)$-dimensional quaternionic space $\mathbb{H}^{n+1}$ as a right $\mathbb{H}$-module with the hermitian form:
$$\langle p,q\rangle_{\mathbb{H}}:=\sum_{l=0}^{n}\overline{p_l}\cdot q_l$$
for $p=(p_0,\cdots,p_n),q=(q_1,\cdots,q_n)\in\mathbb{H}^{n+1}.$ The real part of this hermitian form which we denote by $\langle\cdot,\cdot\rangle$, is the usual  real inner product on $\mathbb{H}^{n+1}$, corresponding to the identification $\mathbb{H}^{n+1}\cong \mathbb{R}^{4(n+1)}$.
\par 
Let us consider the sphere $\mathbb{S}^{7}$ embedded into $\mathbb{H}^{2}$ as the set of elements of norm $1$:
$$\mathbb{S}^{7}= \{q=(q_0,q_1)\in\mathbb{H}^{2}:\|q_0\|_{\mathbb{H}}^2+\|q_1\|_{\mathbb{H}}^2=1\}.$$
There is a natural  diagonal right action of $\mathbb{S}^3$ on $\mathbb{S}^{7}$ which induces the quaternionic Hopf fibration:
$$\mathbb{S}^3\longrightarrow\mathbb{S}^{7}\longrightarrow\mathbb{S}^4.$$
The {\it quaternionic Hopf distribution} $\mathcal{H}_Q$ is the corank $3$ connection of this $\mathbb{S}^3$-principal bundle.  It is given by the orthogonal complement to the following orthonormal vector fields  induced by the right multiplication with the curves $e^{t{\bf i}}, e^{t{\bf j}}, e^{t{\bf k}}$: 
$$V_{\bf i}(q)=-y_0\partial_{x_0}+x_0\partial_{y_0}+\omega_0\partial_{z_0}-z_0\partial_{\omega_0}-y_1\partial_{x_1}+x_1\partial_{y_1}+\omega_1\partial_{z_1}-z_1\partial_{\omega_1}$$
$$V_{\bf j}(q)=-z_0\partial_{x_0}-\omega_0\partial_{y_0}+x_0\partial_{z_0}+y_0\partial_{\omega_0}-z_1\partial_{x_1}-\omega_1\partial_{y_1}+x_1\partial_{z_1}+y_1\partial_{\omega_1}$$
$$V_{\bf k}(q)=-\omega_0\partial_{x_0}+z_0\partial_{y_0}-y_0\partial_{z_0}+x_0\partial_{\omega_0}-\omega_1\partial_{x_1}+z_1\partial_{y_1}-y_1\partial_{z_1}+x_1\partial_{\omega_1}$$
at each $q=(x_0,y_0,z_0,\omega_0,x_1,y_1,z_1,\omega_1)\in\mathbb{S}^{7}$ and with respect to the standard Riemannian metric of $\mathbb{S}^{7}$.
\vspace{1mm} \par
As is well-known the quaternionic Hopf distribution $\mathcal{H}_Q$ is bracket generating,  \cite{BFI0, MaMo12,MaMo11}. Moreover, if we endow $\mathcal{H}_Q$ with the pointwise inner product obtained by restriction from the standard Riemannian metric we obtain a subriemannian structure on $\mathbb{S}^{7}$  which we call 
{\it quaternionic contact structure}. In the following, we write $\mathbb{S}^{7}_Q$ for the sphere $\mathbb{S}^{7}$ endowed with this subriemannian structure. 
\ \par
Note that $\mathbb{S}^{7}_Q$ can also be considered as a quaternionic contact manifold as follows. Let $\eta_{\bf i},\eta_{\bf j},\eta_{\bf k}$ denote the dual frame of the frame $V_{\bf i},V_{\bf j},V_{\bf k}$. Then the quaternionic Hopf distribution $\mathcal{H}_Q$ is locally given by 
$$\mathcal{H}_Q=\bigcap_{{\bf l}\in\{{\bf i},{\bf j},{\bf k}\}}\text{Ker}(\eta_{\bf l}).$$
\par
Furthermore, if we denote by $I_{\bf i},I_{\bf j},I_{\bf k}$ the left multiplications by ${\bf i},{\bf j},{\bf k}$, then it is known that $\{I_{\bf l}:{\bf l}\in\{{\bf i},{\bf j},{\bf k}\}\}$ are almost complex structures satisfying the quaternionic relations compatible with the metric on $\mathcal{H}_Q$, i.e.
$$2\langle I_{\bf l}X,Y\rangle=d\eta_{\bf l}(X,Y)$$
for $X,Y\in\mathcal{H}_Q$ and ${\bf l}\in\{{\bf i},{\bf j},{\bf k}\}$.\ \\

 Recall that the symplectic group ${\bf Sp}(2)$ is the subgroup of $\mathbb{H}$-linear elements of the orthogonal group ${\bf O}(8)$ which preserve the quaternionic inner product. Note that this is a subgroup of the group of all subriemannian isometries  $\mathcal{I}(\mathbb{S}^{7}_Q)$ of $\mathbb{S}^{7}_Q$. 
Hence, by representing elements of ${\bf Sp}(2)$ as $2\times 2$ matrices whose rows build an $\mathbb{H}$-orthonormal basis of $\mathbb{H}^{2}$, we see that ${\bf Sp}(2)$ (and hence $\mathcal{I}(\mathbb{S}^{7}_Q)$) acts transitively on $\mathbb{S}^{7}$.\ \\

The tangent bundle of the sphere $\mathbb{S}^7$ and the orthogonal complement of the quaternionic Hopf distribution $\mathcal{H}_Q$ in $T\mathbb{S}^7$ are both trivial as vector bundles. Hence it is natural to ask wether $\mathcal{H}_Q$ is trivial itself or whether $\mathcal{H}_Q$ admits  at least one  globally defined and nowhere vanishing 
smooth vector field.  In fact, this question was posed as an open problem in \cite[p. 1018]{MaMo11} and will be answered below.\ \\

Given a globally defined smooth vector field $X$ on $\mathbb{S}^{7}$, we consider it as a  smooth function $X: \mathbb{S}^{7}\longrightarrow\mathbb{H}^{2}$ 
such that $$\langle q,X(q)\rangle=0\hspace{3ex}  \text{ \it for all } \hspace{3ex} q\in\mathbb{S}^{7}.$$
\begin{defn}
Let $X$ be a globally defined vector field on $\mathbb{S}^{7}$. We call $X$ a {\it quaternionic vector field} on $\mathbb{S}^{7}$ if $\langle q,X(q)\rangle_\mathbb{H}=0$ 
for all $q\in\mathbb{S}^{7}$. 
\end{defn}
The next lemma states that the quaternionic Hopf distribution is precisely quaternionic tangent space of the sphere:
\begin{lem}\label{quat}
Horizontal vector fields on $\mathbb{S}^{7}$ are the quaternionic vector fields.
\end{lem}
\begin{proof}
By definition, a vector field $X$ on $\mathbb{S}^{7}$ is horizontal if and only if for all $q \in \mathbb{S}^7$: 
$$\langle q,X(q)\rangle=\langle V_{\bf i}(q),X(q)\rangle=\langle V_{\bf j}(q),X(q)\rangle=\langle V_{\bf k}(q),X(q)\rangle=0.$$
Note that the components of the vector fields $V_{\bf i},V_{\bf j}$ and $V_{\bf k}$ at a point $q$  coincide with the components of 
$q {\bf i}, q {\bf j}$ and $q{\bf k}$. A straightforward calculation shows that for $p,q\in\mathbb{H}^{2}$:
$$\langle p,q\rangle_\mathbb{H}=\langle p,q\rangle+{\bf i}\langle p {\bf i},q\rangle+{\bf j}\langle p {\bf j},q\rangle+{\bf k}\langle p {\bf k},q\rangle.$$
This implies that $X$ is horizontal if and only if $\langle q,X(q)\rangle_\mathbb{H}=0$ for all $q\in\mathbb{S}^{7}$, i.e. $X$ is horizontal if and only if $X$ is a quaternionic vector field.
\end{proof}
Now we recall the following quaternionic version of Adam's theorem  in \cite{A} on the maximal dimension of a trivial subbundle of the tangent bundle of a sphere. 
Theorem \ref{Theorem_extension_Adams} below was proven in \cite{Mah} from methods in topological $K$-theory. 
\begin{thm}[\cite{Mah}] \label{Theorem_extension_Adams}
For $n\geq 1$, the sphere $\mathbb{S}^{4n+3}$ admits a nowhere vanishing and globally defined quaternionic vector field if and only if $n\equiv -1$ \textup{mod} $24.$
\end{thm}
By combining this result with Lemma \ref{quat} we obtain:
\begin{cor}\label{loc}
The quaternionic Hopf distribution $\mathcal{H}_Q$ on $\mathbb{S}^7$  does not  admit a nowhere vanishing and globally defined vector field (section of the bundle). In particular, the distribution $\mathcal{H}_Q$ is not trivial.
\end{cor}
\section{Trivializable subriemannian structure}\label{trist}
In the following we recall the definition of a second remarkable subriemannian structure on $\mathbb{S}^7$, called trivializable subriemannian structure 
\cite{BFI,BT}. According to \cite[Theorem 4.4]{BFI} such structures only exist on the spheres $\mathbb{S}^3,\mathbb{S}^7$ and $\mathbb{S}^{15}$.\ \\

 By $\mathbb{K}(n)$ with $\mathbb{K} \in \{ \mathbb{R}, \mathbb{C}, \mathbb{H}\}$ we denote the space of all $n\times n$-matrices with entries in $\mathbb{K}$. 
Let $A_1,\cdots,A_m\in\mathbb{R}(8)$ be a family of skew-symmetric real matrices that fulfill the  anti-commutation relation:
\begin{equation}\label{antic}
A_iA_j+A_jA_i=-2\delta_{ij}\hspace{2ex} \text{ \it for }\hspace{2ex} i,j=1,\dots,m.
\end{equation}

Then a collection of $m$ linear vector fields $X(A_1),\dots,X(A_m)$ on $\mathbb{S}^7$ orthonormal at each point  ({\it canonical vector fields}) 
can be defined in global coordinates of $\mathbb{R}^8$ by: 
$$X(A_k):=\sum_{i,j=1}^{8}(A_k)_{ij}x_j\frac{\partial}{\partial x_i}\hspace{2ex} \text{ \it for }\hspace{2ex} k=1,\dots,m.$$
 Due  %
to the representation theory for Clifford algebras, the  maximal number $m$ of matrices in $\mathbb{R}(8)$ such that the relations (\ref{antic}) hold is $m=7$. 
We recall the following properties  of the above linear vector fields on spheres. 
\begin{lem}[\cite{BFI}]{\label{bra}} 
 Let $A_1, \ldots, A_7 \in \mathbb{R}(8)$ be a collection of matrices with (\ref{antic}). 
For $i=1,\cdots,7$ we set
$$X_j:=X(A_j).$$ 
Then it holds:
\begin{enumerate}
\item For $i,j=1,\cdots,7$  with $i\ne j$: 
$$[X_i,X_j]=-X([A_i,A_j])=-2X(A_iA_j).$$
\item All higher Lie brackets $[X_{i_1}[X_{i_2},[X_{i_3},\dots]]]$ are contained in
$$\textup{Span}\Big{\{}X_i,[X_j,X_k]:i,j,k=1,\cdots,7\Big{\}}.$$
\item Let $i_1,i_2,i_3,i_4\in\{1,\cdots,7\}$. The rank-4 distribution $\mathcal{H}$ on $\mathbb{S}^7$ generated by the vector fields 
$\{X_{i_1},X_{i_2},X_{i_3},X_{i_4}\}$ is bracket generating of step two. 
\end{enumerate}
\end{lem}
\begin{rem}\label{Remark_different_choices_of_generators}
{Let $\{A_1^{(1)},\dots,A_4^{(1)}\}$ and $\{A_1^{(2)},\dots,A_4^{(2)}\}$ be two families of skew-symmetric and anti-commuting matrices in $\mathbb{R}(8)$. Then it was shown in \cite{BFI} that there is $C\in{\bf O}(8)$ such that
$$A_i^{(1)}=C^{-1}A_i^{(2)}C\hspace{3ex} \text{\it  for }\hspace{3ex} i=1,\dots,4.$$
Therefore, if we define the following bracket generating distributions:
$$\mathcal{H}^{(k)}:=\text{Span}\{X(A_i^{(k)}):i=1,\dots,4\}\hspace{2ex} \text{\it  for }\hspace{2ex} k=1,2$$
then the subriemannian structures $(\mathbb{S}^7,\mathcal{H}^{(k)},\langle \cdot,\cdot\rangle)$ for $k=1,2$ are isometric, i.e. the above defined trivializable 
subriemannian structure on $\mathbb{S}^7$ is,  up to subriemannian isometries, independent of the choice of linear vector fields induced by the Clifford module structure of $\mathbb{R}^8$ and spanning the distribution.}
\end{rem}

In the following we give an explicit family of skew-symmetric and anti-commuting matrices which will serve as a model for the study of a trivializable subriemannian 
structure on $\mathbb{S}^7$ induced by matrices which fulfill the relations (\ref{antic}). 
Consider $A_4,A_5,A_6,A_7 \in \mathbb{H}(2)$ defined by: 
\begin{align}\label{matrices_A_j}
A_4:&=
\left(
\begin{array}{cc}
0 & 1\\
-1 & 0 
\end{array}
\right), \hspace{1ex} 
A_5:= 
\left(
\begin{array}{cc}
{\bf i} & 0 \\
0 & -{\bf i} 
\end{array}
\right),\hspace{1ex}
A_6:=
\left(
\begin{array}{cc}
{\bf j} & 0 \\
0 & -{\bf j}
\end{array}
\right) \\
A_7:&=
\left(
\begin{array}{cc}
{\bf k} & 0\\
0 & -{\bf k}
\end{array}
\right).  \notag
\end{align}
One easily verifies that $\{A_4,A_5,A_6,A_7\} \subset \mathbb{H}(2)$ are anti-commuting and skew-symmetric with respect to the standard inner product on $\mathbb{H}^2$.
\vspace{1mm}\\
Via the standard basis of $\mathbb{R}^8$ we may regard $A_j$ as skew-symmetric element in $\mathbb{R}(8)$. 
\begin{lem}\label{Lemma_complement_to_anti-commuting_matrices}
There are  three skew-symmetric matrices $A_1,A_2,A_3 \in \mathbb{R}(8)$ such that $\{A_j \: : \: j=1, \ldots, 7\}\subset \mathbb{R}(8)$ are anti-commuting and skew-symmetric. 
\end{lem}
\begin{proof}
Consider the following skew-symmetric real matrices: 
\begin{align*}
B_1:=
\left(
\begin{array}{cccc}
0 & 0 & 0 & -1 \\
0 & 0 & 1 & 0 \\
0 & -1 & 0 & 0 \\
1 & 0 & 0 & 0
\end{array}
\right), \hspace{1ex} 
B_2:=
\left(
\begin{array}{cccc}
0 & 0 & -1 & 0 \\
0 & 0 & 0 & -1 \\
1 & 0 & 0 & 0 \\
0 & 1 & 0 & 0
\end{array}
\right),\hspace{1ex} 
B_3:=
\left(
\begin{array}{cccc}
0 & 1 & 0 & 0 \\
-1 & 0 & 0 & 0 \\
0 & 0 & 0 & -1 \\
0 & 0 & 1 & 0
\end{array}
\right). 
\end{align*}
Note that $B_1$ (resp. $B_2$ and $B_3$) corresponds to the right quaternionic multiplication by ${\bf k}$ (resp. ${\bf j}$ and ${\bf -i}$). 
Now we define for $i=1,2,3$: 
\begin{align*}
A_i:= 
\left(
\begin{array}{cc}
0 & B_i \\
B_i & 0
\end{array} 
\right). 
\end{align*}
A straightforward calculation shows that the following relations hold:
$$[B_i, {\bf l}]= 0 \hspace{2ex} \text{\it  for }\hspace{2ex} {\bf l} \in \{ {\bf i}, {\bf j}, {\bf k}\}\hspace{2ex} \text{\it  and } \hspace{2ex}i=1,2,3$$
and 
$$B_iB_j+B_jB_i=-2\delta_{ij} \hspace{2ex}  \text{ \it for} \hspace{2ex}  i, j\in\{1,2,3\}.$$
By a direct calculation based on these relations it follows that $A_1, \ldots, A_7$ have the desired properties.
\end{proof} 
We consider the following trivializable distribution on $\mathbb{S}^7$:
$$\mathcal{H}_T:=\text{Span}\{X(A_i):i=1,2,3,4\},$$ 
and we denote by $\mathbb{S}^7_T$ the  trivializable subriemannian manifold $(\mathbb{S}^7,\mathcal{H}_T,\langle \cdot,\cdot\rangle)$ 
 where $\langle \cdot, \cdot \rangle$ denotes the restriction of the standard Riemannian metric on $\mathbb{S}^7$ to the trivial 
bundle $\mathcal{H}_T$.
\begin{rem}
According to Corollary \ref{loc}, the quaternionic Hopf structure $\mathbb{S}^7_Q$ does not admit globally defined and  nowhere vanishing horizontal vector fields and hence it cannot be isometric (as a subriemannian manifold) to the trivializable structure $\mathbb{S}^7_T$.  We will see that both structures not even are locally isometric.
\end{rem}
\section{The Popp measures}\label{pop}
 Recall that the {\it Popp measure} on $\mathbb{S}^7$ is a smooth measure which intrinsically can be assigned to a given regular subriemannian structure (see 
\cite{ABGF,Bariz,BFI-1,Mo}). In the present section we determine the Popp measures $\mathcal{P}_Q$ and $\mathcal{P}_T$ on $\mathbb{S}^7$ corresponding to the quaternionic and the trivializable subriemannian structure, respectively. 
\vspace{1ex}\par
Let $X_1,\cdots,X_4$ be a local  orthonormal frame for the distribution $\mathcal{H}_Q$. Then an adapted frame for $\mathbb{S}^7_Q$ is given by 
$\mathcal{F}=[X_1,\cdots,X_4,V_{\bf i},V_{\bf j},V_{\bf k}]$. According to \cite[Theorem 1]{Bariz} the Popp measure $\mathcal{P}_Q$ for the quaternionic subriemannian structure can be expressed in the form:  
\begin{equation}\label{po1}
\mathcal{P}_Q(z)=\frac{1}{\sqrt{\det{B_Q(z)}}}\eta_1\wedge \cdots\wedge \eta_7, \hspace{5ex} z \in \mathbb{S}^7. 
\end{equation}
\par 
 Here $B_Q(z)$ is  a certain matrix which is obtained from the adapted structure constants of the geometric structure and $\eta_1,\cdots,\eta_7$ denotes 
the dual basis to the frame $\mathcal{F}$ (see \cite{Bariz} for more details). \ \par 
Since the vector fields $X_1,\cdots,X_4,V_{\bf i},V_{\bf j},V_{\bf k}$ are orthonormal with respect to the standard Riemannian metric on $\mathbb{S}^7$, the volume form
$$d\sigma:=\eta_1\wedge \cdots\wedge \eta_7$$
is the standard volume form on $\mathbb{S}^7$.
\begin{lem}\label{Popp_volume_QSS}
The Popp volume $\mathcal{P}_Q$ for the quaternionic structure equals the standard volume form $d\sigma$ up to a constant factor.
\end{lem} 
\begin{proof}
According to (\ref{po1}) we can write $$\mathcal{P}_Q(z)=f(z)d\sigma(z)$$
with  a nowhere vanishing function $f\in C(\mathbb{S}^7)$. We know that the symplectic group ${\bf Sp}(2)$ is a subgroup of the isometry group $\mathcal{I}(\mathbb{S}^7_Q)$. But ${\bf Sp}(2)$ is also a subgroup of ${\bf O}(8)$ which is the isometry group of $\mathbb{S}^7$ with  respect to the standard Riemannian metric. It follows that the Popp volume $\mathcal{P}_Q$  \cite[Proposition 7]{Bariz}  and the standard volume $d\sigma$ are invariant under ${\bf Sp}(2)$, and therefore $f$ must be also invariant under the action of ${\bf Sp}(2)$. Now, the assumption follows from the fact that ${\bf Sp}(2)$ acts transitively on $\mathbb{S}^7$.
\end{proof}
Contrary to the quaternionic Hopf structure, we do not have enough information about the isometry group of the trivializable structure $\mathbb{S}^7_T$ 
 to conclude in a similar way. Therefore, we compute the Popp volume $\mathcal{P}_T$ directly using the  adapted structure constants.\ 
An adapted frame for the trivializable structure is given globally by the orthonormal vector fields $X_1,\cdots,X_7$ defined from the matrices $A_1,\cdots,A_7$ in Lemma \ref{Lemma_complement_to_anti-commuting_matrices}.  According to  \cite[Theorem 1]{Bariz} the Popp measure can be written as
$$\mathcal{P}_T(z)= \frac{1}{\sqrt{\det B_T(z)}}d\sigma(z),$$
where $B_T(z)=(B^{kl}_T(z))_{k,l=5}^7$ is the $3\times 3$ matrix function on $\mathbb{S}^7$ with coefficients
$$B^{kl}_T(z)=\sum_{i,j=1}^{4}b_{ij}^k(z)b_{ij}^l(z), \hspace{4ex} z\in\mathbb{S}^7.$$
For  $i,j=1,\dots,4$ and $k=5,6,7$ the functions $b_{ij}^k(z)$ are defined by:
\begin{equation}\label{Definition_b_i_j_l}
b_{ij}^k(z)=\langle [X_i,X_j](z),X_k(z)\rangle=-2\langle A_iA_j z,A_k z\rangle \hspace{2ex} \text{ \it for }\hspace{2ex} z\in\mathbb{S}^7.
\end{equation}
\par 
In (\ref{Definition_b_i_j_l}) we have used the  notation $\langle \cdot, \cdot \rangle$ for the Euclidean inner product on $\mathbb{R}^8$ and 
its restriction to the sphere, respectively. In the following, we write $\|A\|_{\textup{HS}}$ for the Hilbert-Schmidt norm of $A\in\mathbb{R}(8)$.
\begin{lem}\label{popp}
The Popp measure $\mathcal{P}_T$ with respect to the trivializable subriemannian structure $\mathbb{S}^7_T$ is given by
$$\mathcal{P}_T(z)=g(z)d\sigma,$$
where $$g(z):=\left[16(1-2\|x\|^2\|y\|^2)\right]^{-3/2}\hspace{2ex} \text{ for }\hspace{2ex} z=(x,y)\in\mathbb{S}^7\subset \mathbb{R}^8.$$
\end{lem}
\begin{proof}
We introduce the following notations:
$$A_{\bf i}:=A_5,\hspace{1mm}A_{\bf j}:=A_6,\hspace{1mm}A_{\bf k}:=A_7\hspace{2ex} \text{\it  and }\hspace{2ex} A_8:=Id.$$
Let $l\in\{5,6,7\}$ and $z=(x,y)\in\mathbb{S}^7$.  
{Using the fact that the  skew-symmetric and anti-commuting matrices $A_1,\cdots,A_7$ lie in ${\bf O}(8)$ and that $\{A_1z,\cdots,A_8z\}$ forms an orthonormal basis of $\mathbb{R}^8$, we can write:}
\begin{align*}
B_{T}^{ll}(z)
&= 4 \cdot  \sum_{i,j=1}^4 \big{\langle} A_lA_{i}z, A_{j}z \big{\rangle}^2\\
&= 4 \cdot \left(\|A_l\|_{\textup{HS}}^2 - \sum_{i=5}^8 \sum_{j=1}^8 \big{\langle} A_lA_{i}z,A_{j} z\big{\rangle}^2- 
 \sum_{i=1}^4 \sum_{j=5}^8 \big{\langle} A_l A_{i}z, A_{j}z \big{\rangle}^2 \right)\\
&=4 \cdot \left( \|A_l\|_{\textup{HS}}^2 - \sum_{i=5}^8\underbrace{ \big{\|}A_lA_{i}z\big{\|}^2}_{=1}- \sum_{i=1}^4 \sum_{j=5}^7 
\big{\langle} A_l A_{i}z, A_{j}z \big{\rangle}^2 \right).
\end{align*}
Furthermore, a straightforward calculation shows that for ${\bf l}\neq{\bf m}\in\{{\bf i},{\bf j},{\bf k}\}$:
$$|\langle A_{\bf l}A_iz,A_{\bf m}z\rangle|=2|\langle B_ix,({\bf l}\cdot{\bf m})y\rangle|\hspace{2ex} \text{\it  for } \hspace{2ex} i=1,\cdots,4.$$
Here $B_1,B_2$ and $B_3$ are the matrices defined in Lemma \ref{Lemma_complement_to_anti-commuting_matrices} and $B_4:=Id$.

We assume that $x \ne 0$. Since $\{\|x\|^{-1} B_ix:i=1,\cdots,4\}$ is an orthonormal basis of $\mathbb{H} \cong \mathbb{R}^4$ it follows that for ${\bf l},{\bf m}\in\{{\bf i},{\bf j},{\bf k}\}$:
\begin{equation*}
\sum_{i=1}^4 \langle A_{\bf l}A_iz, A_{\bf m}z\rangle^2= 4 \|x\|^2\| ({\bf l}\cdot{\bf m})y\|^2= 4\|x\|^2\|y\|^2. 
\end{equation*}
Equality also holds in the case $x=0$.
Therefore, we find for $l=5,6,7$: 
\begin{equation*}
B_T^{ll}(z)=4(4-8 \|x\|^2\|y\|^2)= 16(1-2\|x\|^2\|y\|^2). 
\end{equation*}
For $l\neq m\in\{5,6,7\}$ it holds:
\begin{align*}
\frac{1}{4} B_T^{lm}(z)
&=\sum_{i_1,i_2=1}^4 \big{\langle} A_lA_{i_1}z, A_{i_2}z \big{\rangle} \big{\langle} A_mA_{i_1}z,A_{i_2}z \big{\rangle}\\
&= \Big{(}\sum_{i_1,i_2=1}^8 -\sum_{i_1=5}^8 \sum_{i_2=1}^8 - \sum_{i_1=1}^4 \sum_{i_2=5}^8 \Big{)}  
\big{\langle} A_lA_{i_1}z, A_{i_2}z \big{\rangle} \big{\langle} A_mA_{i_1}z,A_{i_2}z \big{\rangle}\\
&= \sum_{i_1=1}^8 \underbrace{\big{\langle} A_mA_{i_1}z, A_lA_{i_1}z \big{\rangle}}_{=0}- \sum_{i_1=5}^8  \underbrace{\big{\langle} A_mA_{i_1}z, A_lA_{i_1}z \big{\rangle}}_{=0}\\
&\hspace{5ex} - \sum_{i_1=1}^4\sum_{i_2=5}^8\big{\langle} A_lA_{i_1}z, A_{i_2}z \big{\rangle} \big{\langle} A_mA_{i_1}z,A_{i_2}z \big{\rangle}. 
\end{align*}
Since the matrices $A_1,\cdots,A_7$ are anti-commuting, it follows that
$$\big{\langle} A_lA_{i_1}z, A_{i_2}z \big{\rangle} \big{\langle} A_mA_{i_1}z,A_{i_2}z \big{\rangle}=0\hspace{2ex} \text{\it  for }\hspace{2ex} i_2\in\{l,m\}.$$
Hence we can write with $i_2\in\{5,6,7\}\backslash \{l,m\}$ {and ${\bf i_2} \in \{ {\bf i}, {\bf j}, {\bf k} \}$ defined by $A_{i_2}= A_{\bf i_2}$: }
\begin{align*}
\frac{1}{4} B_T^{lm}(z)
&=- \sum_{i_1=1}^4\big{\langle} A_lA_{i_1}z, A_{i_2}z \big{\rangle} \big{\langle} A_mA_{i_1}z,A_{i_2}z \big{\rangle}
\\
&= -4 \sum_{Q \in \{ I, B_1,B_2,B_3\}} \big{\langle} Qx, ({\bf l\cdot i_2})y \big{\rangle} \big{\langle} Qx, ({\bf m\cdot i_2})y \big{\rangle}\\
&=-4 \big{\langle} ({\bf l\cdot i_2})y, ({\bf m\cdot i_2})y \big{\rangle}\\
&= -4 \big{\langle} {\bf l}y,{\bf m}y\big{\rangle}=0. 
\end{align*}
We obtain: 
\begin{equation}\label{Matrix_B_T}
 B_T(z)= 16(1-2\|x\|^2\|y\|^2) \cdot \textup{Id} \in \mathbb{R}(3) 
\end{equation}
and therefore, the Popp measure $\mathcal{P}_T$ has the form: 
\begin{equation*}
\mathcal{P}_T(z)=\big{[} 16(1-2\|x\|^2\|y\|^2) \big{]}^{- \frac{3}{2}} d\sigma. 
\end{equation*}
\end{proof}
\section{The nilpotent approximation}\label{nilapp}
Let $z=(x,y)\in\mathbb{S}^7\subset \mathbb{R}^8$. Since $X_1,\cdots,X_7$ is an adapted orthonormal frame for $\mathbb{S}^7_T$, the tangent algebra at $z$ for $\mathbb{S}^7_T$ is the Carnot algebra of step $2$ given by
\begin{equation}\label{nil1}
\mathfrak{g}_z=\mathcal{H}_z\oplus \mathcal{V}_z\simeq \mathbb{R}^7,
\end{equation}
where 
\begin{align*}
\mathcal{H}_z:&=\textup{Span}\{X_i(z):i=1,\cdots,4\}, \\
\mathcal{V}_z:&=\textup{Span}\{X_k(z):k=5,6,7\}.
\end{align*}
For $i,j=1,\cdots,4$ the Lie brackets are given by: 
$$[X_i(z),X_j(z)]:=\sum_{k=5}^{7}\langle [X_i,X_j],X_k\rangle_z X_k(z).$$
 Note that the inner product $\langle \cdot,\cdot\rangle_z$ on $\mathcal{H}_z$ induces an inner product on  the first layer of the graded Lie algebra  $\mathfrak{g}_z$, i.e. $\mathfrak{g}_z$ is a Carnot Lie algebra.

In the following, we need a technical lemma on the local comparison of two subriemannian manifolds. First, we recall the definition of a 
{\it nonsingular Carnot algebra}, see \cite{Eber,Gornet} for more details.\ \\

Let $\mathfrak{g}=\mathfrak{g}_1\oplus\mathfrak{g}_2$ be a Carnot algebra of step $2$, i.e.
$$[\mathfrak{g}_1,\mathfrak{g}_1]=\mathfrak{g}_2 \hspace{2ex} \text{ \it and }\hspace{2ex} [\mathfrak{g}_i,\mathfrak{g}_j]=\{0\}\text{ for }i+j>2.$$
We assume that an inner product $\langle\cdot,\cdot\rangle$ on $\mathfrak{g}_1$ is given. Then every element $Z\in\mathfrak{g}_2^\ast$ induces a representation map $J_Z:\mathfrak{g}_1\longrightarrow\mathfrak{g}_1$ defined by
$$\langle J_Z X,Y\rangle:=Z([X,Y])\hspace{2ex} \text{ \it for }\hspace{2ex} X,Y\in\mathfrak{g}_1.$$
\begin{defn}
We say that the Carnot algebra $(\mathfrak{g},\langle\cdot,\cdot\rangle)$ is {\it nonsingular}, if for all $Z\in\mathfrak{g}_2^\ast\backslash\{0\}$, the induced map $J_Z$ is invertible. Otherwise,  $(\mathfrak{g},\langle\cdot,\cdot\rangle)$ is called {\it singular}.
\end{defn}
Note that if $\varphi:(\mathfrak{g},\langle\cdot,\cdot\rangle)\rightarrow(\mathfrak{g}^\prime,\langle\cdot,\cdot\rangle^\prime)$ is a Lie algebra isomorphism which preserves the inner products (i.e. an isometry), then $(\mathfrak{g}^\prime,\langle\cdot,\cdot\rangle^\prime)$ will be nonsingular (resp. singular) if and only 
if $(\mathfrak{g},\langle\cdot,\cdot\rangle)$ is. Hence we obtain:
\begin{lem}\label{noni}
Let $(M,\mathcal{H},g)$ and $(M^\prime,\mathcal{H}^\prime,g^\prime)$ be step two subriemannian manifolds which  near a point $x \in M$
are locally isometric by $\phi$. If the nilpotent approximation of $M$ at $x\in M$ is nonsingular, then so is the nilpotent approximation of $M^\prime$ at $\phi(x)$.
\end{lem}

By considering $\mathbb{S}^7_Q$ as a quaternionic contact manifold, it is easy to see that its tangent algebra can be identified at every point with the quaternionic Heisenberg Lie algebra, which, in particular, is non-singular. For the trivializable subriemannian structure on $\mathbb{S}^7$, the situation is completely different. As we will see, its tangent algebra can be different from point to point.\ \\

Let $\alpha,\beta,\gamma\in\mathbb{R}$ and consider the vertical vector field
$$Z:=\alpha X_5(z)+\beta X_6(z)+\gamma X_7(z)\in\mathcal{V}_z.$$ 
By declaring the vectors $X_5(z),X_6(z),X_7(z)$ to be orthonormal, we obtain an inner product on $\mathcal{V}_z$  which again is denoted by $\langle \cdot, \cdot \rangle$. This induces an identification of $\mathcal{V}_z^\ast$ with $\mathcal{V}_z$ so that we can write for $J_Z:\mathcal{H}_z\longrightarrow\mathcal{H}_z$:
$$\langle J_Z X,Y\rangle_z=\langle Z,[X,Y]\rangle_z \hspace{2ex} \text{\it  for } \hspace{2ex} X,Y\in\mathcal{H}_z.$$
Let $A(\alpha,\beta,\gamma)$ denote the following element of $\mathbb{H}$:
$$A(\alpha,\beta,\gamma):=\alpha{\bf i}+\beta{\bf j}+\gamma{\bf k}.$$
Then a straightforward calculation shows that:
\begin{align*}
\langle Z,[X_1(z),X_2(z)]\rangle_z &=2\langle A(\alpha,\beta,\gamma)x,B_3x\rangle-2\langle A(\alpha,\beta,\gamma)y,B_3y\rangle =a-d\\
\langle Z,[X_1(z),X_3(z)]\rangle_z &=-2\langle A(\alpha,\beta,\gamma)x,B_2x\rangle+2\langle A(\alpha,\beta,\gamma)y,B_2y\rangle =-b+e\\
\langle Z,[X_1(z),X_4(z)]\rangle_z &=2\langle A(\alpha,\beta,\gamma)x,B_1x\rangle+2\langle A(\alpha,\beta,\gamma)y,B_1y\rangle =c+f\\
\langle Z,[X_2(z),X_3(z)]\rangle_z &=2\langle A(\alpha,\beta,\gamma)x,B_1x\rangle-2\langle A(\alpha,\beta,\gamma)y,B_1y\rangle =c-f\\
\langle Z,[X_2(z),X_4(z)]\rangle_z &=2\langle A(\alpha,\beta,\gamma)x,B_2x\rangle+2\langle A(\alpha,\beta,\gamma)y,B_2y\rangle =b+e\\
\langle Z,[X_3(z),X_4(z)]\rangle_z &=2\langle A(\alpha,\beta,\gamma)x,B_3x\rangle+2\langle A(\alpha,\beta,\gamma)y,B_3y\rangle =a+d.
\end{align*}

Hence, with respect to the basis $\{X_i(z):i=1,\cdots,4\}$,  the operator $J_Z$ can be represented by a skew-symmetric matrix of the form: 
\begin{equation}\label{max}
\begin{pmatrix}
 0 & d-a & b-e & -c-f\\
 a-d & 0 & -c+f & -b-e\\
 -b+e & c-f & 0 & -a-d\\
 c+f & b+e & a+d & 0\\
\end{pmatrix}
\end{equation}
with $a,b,c,d,e,f\in\mathbb{R}$ as above.\ \\

Note that  the matrix in (\ref{max}) has the determinant: $$(a^2+b^2+c^2-d^2-e^2-f^2)^2.$$
By using the following identity for $\omega\in\mathbb{R}^4$ :
$$\langle A(\alpha,\beta,\gamma)\omega,C\omega\rangle^2+\langle A(\alpha,\beta,\gamma)\omega,D\omega\rangle^2+\langle A(\alpha,\beta,\gamma)\omega,E\omega\rangle^2=(\alpha^2+\beta^2+\gamma^2)\|\omega\|^4,$$
we calculate the determinant of $J_Z$:
$$ \det J_Z= 16(\|x\|^2-\|y\|^2)^2(\alpha^2+\beta^2+\gamma^2)^2.$$
Hence, if $\|x\|\neq \|y\|$ then the operator $J_Z$ is invertible for all $Z\in\mathcal{V}_z\backslash\{0\}$.
\begin{lem}\label{sing}
Let $z=(x,y)\in\mathbb{S}^7_T$. Then the tangent algebra of $\mathbb{S}^7_T$ at $z$  is non-singular if and only if $\|x\|\neq\|y\|$.
\end{lem}
Using Lemma $\ref{noni}$ and Lemma $\ref{sing}$ we conclude that:
\begin{thm}
The subriemannian manifolds $\mathbb{S}^7_Q$ and $\mathbb{S}^7_T$ are not locally isometric. Furthermore, the isometry group  $\mathcal{I}(\mathbb{S}^7_T)$ of the trivializable subriemannian structure does not act transitively on $\mathbb{S}^7$.
\end{thm}

\section{On the type of distributions}\label{type}
We have seen that the tangent algebras of $\mathbb{S}^7_T$ are  nonsingular outside the set
$$\mathcal{S}:=\{z=(x,y)\in\mathbb{S}^7:\|x\|=\|y\|\}.$$
In the following we show that the trivializable distribution $\mathcal{H}_T$ fails to be elliptic on this singular set $\mathcal{S}$.\ \\

Recall that the curvature map (\ref{curvature_map}) of the distribution $\mathcal{H}_T$ is defined by
\begin{align*}
F:\Lambda^2\mathcal{H}_T&\longrightarrow T\mathbb{S}^7/\mathcal{H}_T\\
(X,Y)&\longmapsto F(X,Y):=-[X,Y] \text{ mod } \mathcal{H}_T.
\end{align*}
The dual curvature map $\omega$ in (\ref{dual_curvature_map}) is then given as the dual map, i.e.
\begin{align*}
\omega:\mathcal{H}_T^\perp&\longrightarrow \Lambda^2\mathcal{H}_T^\ast\\
\lambda &\longmapsto \omega(\lambda),
\end{align*}
with $$\omega(\lambda)(X\wedge Y):=-\lambda([X,Y]) \hspace{2ex} \textup{\it  for all} \hspace{2ex}  X,Y\in\mathcal{H}_T.$$
 Using the standard Riemannian metric on $\mathbb{S}^7$, we identify $\mathcal{H}^\perp_T$ with $$\mathcal{V}:=\text{Span}\{\langle X_j,\cdot\rangle:j=5,6,7\}.$$
 The distribution $\mathcal{H}_T$ is generated by globally defined vector fields $X_1,\cdots,X_4$ and this induces a specific horizontal form, namely
  $$\eta_{\mathcal{H}_T}:=\eta_1\wedge\cdots\wedge\eta_4\in \Lambda^4\mathcal{H}^\ast_T,$$
where $\eta_1,\cdots,\eta_4$  denotes the frame dual to $X_1,\cdots,X_4$. 
Now the dual curvature map $\omega$ induces a  family parametrized over $M$  
of real quadratic forms $Q:=\omega^2/\eta_{\mathcal{H}_T}$ on $\mathcal{H}_T^\perp\simeq\mathcal{V}$ defined by:
\begin{align*}
\omega^2:\mathcal{H}_T^\perp&\longrightarrow \Lambda^4\mathcal{H}^\ast_T\\
\lambda &\longmapsto \omega(\lambda)\wedge\omega(\lambda)=Q(\lambda)\eta_{\mathcal{H}_T}.
\end{align*}
In the following lemma we compute the quadratic form $Q$ for the trivializable subriemannian structure on $\mathbb{S}^7.$
\begin{lem}
Let $\lambda=\sum_{l=5}^{7}\lambda^l X_l\in\mathcal{V}\simeq\mathcal{H}_T^\perp$. Then the quadratic form $Q$ is given by
$$Q(\lambda)=2\sum_{k,l=5}^{7}\left(b_{12}^lb_{34}^k+b_{14}^lb_{23}^k-b_{13}^lb_{24}^k\right)\lambda^l\lambda^k,$$
 where for $i,j =1, \ldots, 4$ the coefficients $b_{ij}^k$ have been defined in (\ref{Definition_b_i_j_l}). 

\end{lem}
\begin{proof}
For $X=\sum_{i=1}^{4}\alpha_i X_i\text{ and }Y=\sum_{j=1}^{4}\beta_j X_j\in\mathcal{H}_T$ it holds:
\begin{align*}
\omega(\lambda)(X\wedge Y)&=-\langle \lambda,[X,Y]\rangle\\
&=-\sum_{i,j=1}^{4}\alpha_i \beta_j\langle \lambda,[X_i,X_j]\rangle\\
&=-\sum_{1\leq i <j\leq 4}(\alpha_i \beta_j-\alpha_j \beta_i)\langle \lambda,[X_i,X_j]\rangle\\
&=-\sum_{1\leq i <j\leq 4}\langle \lambda,[X_i,X_j]\rangle \eta_i\wedge\eta_j(X,Y).
\end{align*}
Hence the dual curvature map $\omega$ is given by :
$$\omega(\lambda)=-\sum_{1\leq i <j\leq 4}\langle \lambda,[X_i,X_j]\rangle \eta_i\wedge\eta_j,$$
with $$\langle \lambda,[X_i,X_j]\rangle={\sum_{l=5}^{7}b_{ij}^l\lambda^l}.$$
A straightforward calculation shows now that
$$\omega(\lambda)^2=\left(2\sum_{k,l=5}^{7}\left(b_{12}^lb_{34}^k+b_{14}^lb_{23}^k-b_{13}^lb_{24}^k\right)\lambda^l\lambda^k\right)\eta_{\mathcal{H}_T}.$$
\end{proof}
We set for $k,l\in\{5,6,7\}$:
$$T^{lk}:=b_{12}^lb_{34}^k+b_{12}^kb_{34}^l+b_{14}^lb_{23}^k+b_{14}^kb_{23}^l-b_{13}^lb_{24}^k-b_{13}^kb_{24}^l.$$
Using similar arguments as for the computation of the Popp volume for $\mathbb{S}^7_T$, we find that the off-diagonal symbols $T^{lk}$ vanish and that
$$T^{11}=T^{22}=T^{33}=2(\|x\|^2-\|y\|^2).$$
Hence, it follows that the quadratic form $Q$ for the trivializable structure $\mathbb{S}^7_T$ is given explicitly by
$$Q(\lambda)=2\sum_{l=5}^{7}(\|x\|^2-\|y\|^2)(\lambda^l)^2.$$ 
\begin{cor}
The trivializable distribution $\mathcal{H}_T$ on $\mathbb{S}^7$ is of elliptic type on the open dense subset $\{(x,y)\in\mathbb{S}^7:\|x\|\neq\|y\|\}$. Otherwise, it is of type $(0,0)$.
\end{cor}
It was shown in \cite{Mo} that every distribution of elliptic type on a $7$-dimensional manifold has a finite dimensional symmetry group of maximal dimension $21$. Furthermore, the sphere $\mathbb{S}^7_Q$ equipped with the quaternionic Hopf distribution $\mathcal{H}_Q$  has a symmetry group of maximal dimension. 
The trivializable structure on $\mathbb{S}^7$ is everywhere elliptic on $\mathbb{S}^7$ except on $\mathcal{S}$ which is a closed submanifold of $\mathbb{S}^7$ of dimension $6$.\ \\
If $\phi:\mathbb{S}^7_T\longrightarrow\mathbb{S}^7_T$ is a diffeomorphism preserving the distribution $\mathcal{H}_T$, then by Lemma \ref{noni} and Lemma \ref{sing}, the submanifold $\mathcal{S}$ must be invariant under $\phi$ and hence $\phi$ restricts to a diffeomorphism $\mathbb{S}^7\backslash\mathcal{S}\longrightarrow\mathbb{S}^7\backslash\mathcal{S}$ preserving the everywhere elliptic distribution $\mathcal{H}_T$ on $\mathbb{S}^7\backslash\mathcal{S}$. Hence the symmetry group of $\mathbb{S}^7_T$ is also finite dimensional with dimension bounded by $21$.
\vspace{1ex} \par
 In the following, by giving a $3$-dimensional family of subriemannian isometries of $\mathbb{S}_T^7$, we show that the isometry group $\mathcal{I}(\mathbb{S}_T^7)$ has dimension greater than or equal to $3$.
 Let $x=(x_0,x_1,x_2,x_3)\in\mathbb{S}^3$ and consider the following matrix
$$C:=\begin{pmatrix}
 x_0 & x_1 & x_2 & x_3\\
 x_3 & -x_2 & x_1 & -x_0\\
 -x_2 & -x_3 & x_0 & x_1\\
 -x_1 & x_0 & x_3 & -x_2\\
\end{pmatrix}\in{\bf O}(4).$$
Then the following relations hold:
\begin{equation}\label{examp}
 B_3C=CB_1, \hspace{2ex}
CB_3=-B_1C\hspace{2ex} \text{\it  and }\hspace{2ex} CB_2=B_2C.
\end{equation}

Let us define the following block matrix in ${\bf O}(8)$:
$$U:=\begin{pmatrix}
0& C\\
C B_1&0\\
\end{pmatrix}.$$
Then based on the relations \eqref{examp} and the commutation relations of the matrices $B_j$ (s. Lemma \ref{Lemma_complement_to_anti-commuting_matrices}) 
we have:
$$UA_1=A_4U\hspace{2ex} \text{ \it and } \hspace{2ex} UA_j=A_{j-1}U\hspace{2ex} \text{\it  for }\hspace{2ex} j=2,3,4.$$
In particular, this imply that $U$ defines a subriemannian isometry of $\mathbb{S}^7_T$.

\section{Small time asymptotics of the heat kernel}\label{small}
An analysis of the intrinsic sublaplacian induced by the quaternionic Hopf structure on $\mathbb{S}^{4n+3}$ was done in \cite{Baud_Wang}. In particular, the first  heat invariants $c_0$ and $c_1$,   i.e. the first two coefficients in the small time asympotic expansion of the heat trace, have been explicitly calculated. In the general setting of subriemannian manifolds, a powerful method in the  analysis of a sublaplacian is given by the so-called {\it nilpotent approximation}. 
 The idea consists in an approximation of the subriemannian manifold at a given point by a nilpotent Lie group endowed with a left-invariant subriemannian structure. In the following, we briefly  recall the relevant concepts. For more details we refer to \cite{Chpo, Verd}.\ \\

Let $(M,\mathcal{H},\langle\cdot,\cdot\rangle)$ be a step two regular subriemannian manifold and by 
$$ \{X_1,\cdots,X_m,X_{m+1},\ldots ,X_n\}$$ 
we denote a local adapted frame at $q\in M$. A system of local coordinates 
$$\psi: M \supset U_q\longrightarrow \mathbb{R}^n=\mathbb{R}^m\oplus\mathbb{R}^{n-m}$$
 is called {\it linearly adapted at $q$}  if 
$$\psi(q)=0 \hspace{2ex} \text{\it  and }\hspace{2ex} \psi_\ast(\mathcal{H}_q)=\mathbb{R}^m.$$
In a system of linearly adapted coordinates at $q$, we have a notion of nonholonomic orders $"ord"$ corresponding to the natural dilations $\delta_\lambda:\mathbb{R}^n\rightarrow\mathbb{R}^n$ defined  for $\lambda>0$ by
$$\delta_\lambda(x_1,\cdots,x_m,x_{m+1},\cdots,x_n):=(\lambda x_1,\cdots,\lambda x_m,\lambda^2 x_{m+1},\cdots,\lambda^2 x_n).$$
More precisely, we set: 
$$\textup{ord}(x_i):=\begin{cases}
1& \text{if } 1\leq i\leq m,\\
2& \text{if } m+1\leq i\leq n 
\end{cases}$$
and 
$$\textup{ord}\left(\frac{\partial}{\partial x_i}\right):=\begin{cases}
-1& \text{if } 1\leq i\leq m,\\
-2& \text{if } m+1\leq i\leq n.
\end{cases}$$
Furthermore, every smooth vector field $X$ on $\mathbb{R}^n$ has an expansion near $0$ of the form:
$$X\simeq X^{(-2)}+X^{(-1)}+\cdots,$$
where $X^{(l)}$ is a polynomial vector field of order $l$, i.e. homogeneous of order $l$ with respect to the dilations $\delta_\lambda$.  A straightforward 
calculation shows the following behaviour of the order function under Lie brackets: 
$$\textup{ord}[X,Y]\leq \textup{ord}(X)+\textup{ord}(Y).$$
In the following we need a special class of linearly adapted coordinates called {\it privileged coordinates}. These are linearly adapted coordinates at $q$ such that every vector field $\psi_\ast(X_i)$ for $1\leq i\leq m$, has an expansion near $0$ where all the homogeneous terms have orders greater than $-1$:
\begin{equation}\label{asymp}
\psi_\ast(X_i)\simeq X_i^{(-1)}+X_i^{(0)}+\cdots.
\end{equation}
An example of privileged coordinates at $q\in M$ is given by the so-called canonical coordinates of the first kind  defined as the inverse of the local diffeomorphism:
$$(x_1,\cdots,x_n)\longmapsto \exp{(x_1X_1+\cdots+x_n X_n)}(q).$$
Here $X_1,\cdots,X_n$ is an adapted local frame at $q$. 
\vspace{1mm}\par 
Note that the vector fields $X_1^{(-1)},\cdots,X_n^{(-1)}$ on $\mathbb{R}^n$ generate a graded step two nilpotent Lie algebra $\widetilde{\mathfrak{g}}(q)$ isomorphic to the tangent Lie algebra $\mathfrak{g}M(q)$ at $q$ (see \cite{Chpo}). Let us denote by $(\widetilde{\mathbb{G}}(q),\ast)$ the induced step two nilpotent Lie group defined as follows. As a manifold we take $\widetilde{\mathbb{G}}(q)=\mathfrak{g}M(q)$ and the group law is defined by
$$\xi_1\ast\xi_2:=\xi_1+\xi_2+\frac{1}{2}[\xi_1,\xi_2] \hspace{2ex} 
\textup{\it for} \hspace{2ex}  \xi_1,\xi_2\in\widetilde{\mathbb{G}}(q).$$
\begin{defn}
Given a smooth measure $\mu$ on $M$, its {\it nilpotentization} at $q$  is a measure $\hat{\mu}^q$ on $\widetilde{\mathbb{G}}(q)$ defined in the chart $\psi$  by
$$\hat{\mu}^q:=\lim_{\epsilon\to 0}\frac{1}{\epsilon^{Q}}\delta_\epsilon^\ast\mu.$$

Here the convergence is understood in the weak-$*$-topology of $C_c(M)^\prime$ and $Q$ denotes the Hausdorff dimension of the regular SR manifold $M$. Due to the regularity assumption of the SR manifold $M$, the measure $\hat{\mu}^q$ is in fact a left-invariant measure on $\widetilde{\mathbb{G}}(q)$ which is nilpotent and hence unimodular. Therefore the measure $\hat{\mu}^q$ is a Haar measure on $\widetilde{\mathbb{G}}(q)$ (see \cite{Verd}).
\end{defn}
Now we recall the relation between the first heat invariant $c_0$ and the nilpotentization of the subriemannian manifold $M$. 
As was mentioned in Section \ref{subgeo}, the heat kernel $K_t(\cdot,\cdot)$  has an asymptotic expansion on the diagonal as $t\to 0$ of the form: 
$${K_t(q,q)}=\frac{1}{t^{Q(q)/2}}\left(c_0(q)+c_1(q)t+\cdots+c_N(q)t^N+o(t^N)\right)$$
for all $N\in\mathbb{N}$ and $q\in M$. Here the {(locally defined)} smooth coefficients $c_i(q)$ are called {\it heat invariants} of the SR manifold $M$.
\vspace{1ex}\par
Let ${K_t^{\widetilde{\mathbb{G}}(q)}}$ denote the heat kernel of the sublaplacian 
$$\Delta_{\textup{sub}}^{\widetilde{\mathbb{G}}(q)}={-}\sum_{i=1}^{m}\left(X_i^{(-1)}\right)^2$$
on $\widetilde{\mathbb{G}}(q)$ with respect to the Haar measure $\hat{\mu}^q$. 
 According to the results in \cite{Verd} the first heat invariant $c_0$ is given by: 
\begin{equation}\label{firstheat}
c_0(q)={K_t^{\widetilde{\mathbb{G}}(q)}(1,0,0).}
\end{equation}
In general, calculating the remaining heat invariants $c_1,c_2,\cdots$ with the help of  the nilpotentization is rather complicated. However, 
in the special case where the horizontal frame is $\mu$-divergence free, the sublaplacian $\Delta_{\textup{sub}}^\mu$ is a sum of squares, i.e
$$\Delta_{\textup{sub}}^\mu=-\sum_{i=1}^{m}X_i^2$$
and the formula for $c_1$ simplifies to {(see the proof of Theorem A in \cite{Verd})}
\begin{equation}\label{second_invariant}
c_1(q)=\int_{0}^{1}\int_{\mathbb{R}^n}K_s^{\widetilde{\mathbb{G}}(q)}(0,\xi)Y\left( K_{1-s}^{\widetilde{\mathbb{G}}(q)}(\xi,0)\right)d\xi ds,
\end{equation}
where $Y$ is a second order differential operator acting with respect to the variable $\xi\in\mathbb{R}^n$. More precisely, it is given by
$$Y:=\sum_{i=1}^{m}X_i^{(-1)}X_i^{(1)}+X_i^{(1)}X_i^{(-1)}+X_i^{(0)}X_i^{(0)}.$$

For the trivializable subriemannian structure on $\mathbb{S}^7$ we have two choices of a natural smooth measure. The first one is the measure induced by the standard Riemannian metric on $\mathbb{S}^7$ which we denote by $d\sigma$ and the second one is the Popp measure $\mathcal{P}_T$. 
The sublaplacian with respect to the Popp measure can be expressed as  (see \cite{ABGF,Bariz}): 
$$\Delta_{\textup{sub}}^{T}=-\sum_{i=1}^{4}\left(X_i^2+\text{div}_{\mathcal{P}_T}(X_i)X_i\right).$$
Here $\{X_1,\cdots,X_4\}$ denotes the globally defined orthonormal frame of $\mathcal{H}_T$. We recall that by Lemma \ref{popp} the Popp measure is given by
 $$\mathcal{P}_T(z)=g(z)d\sigma(z) \hspace{2ex}\mbox{\it where } \hspace{2ex} g(z)=\big{(} 16(1-2\|x\|^2 \|y\|^2) \big{)}^{-\frac{3}{2}}, \hspace{3ex} z=(x,y). $$
 Therefore, using the fact that $X_1,\cdots,X_4$ are Killing vector fields and hence $\sigma$-divergence free and  by using the formula
$$\text{div}_{\mathcal{P}_T}(X)=\text{div}_{\sigma}(X)+X(\log{g})$$
for a smooth vector field $X$ on $M$, we see that
$$\text{div}_{\mathcal{P}_T}(X_i)=X_i(h)\hspace{2ex} \text{ \it for } \hspace{2ex} i=1,\cdots,4$$
with
$$h(z):=-\frac{3}{2}\log{(1-2\|x\|^2\|y\|^2)} \hspace{3ex} \text{ \it for }\hspace{3ex}  z=(x,y)\in\mathbb{S}^7.$$ 
Hence we have the following formula: 
 \begin{lem}\label{intrinsic_sublaplacian_s_7}
The intrinsic sublaplacian $\Delta_{\textup{sub}}^{T}$ on the trivializable SR manifold $\mathbb{S}_T^7$ acting on $C^\infty(\mathbb{S}^7)$ is given by: 
$$\Delta_{\textup{sub}}^{T}=-\sum_{i=1}^{4}\left(X_i^2+X_i(h)X_i\right).$$
\end{lem}
\begin{rem} 
\label{Remark_Isospectrality_of_equivalent_triv_structures}
A different choice of the anti-commuting skew-symmetric matrices $A_j$ leads to a subriemannian structure on $\mathbb{S}^7$ (with intrinsic sublaplacian $\Delta_{\textup{sub}}^{T^\prime}$) equivalent to $\mathbb{S}^7_T$ (s. Remark \ref{Remark_different_choices_of_generators}). Furthermore, a subriemannian isometry preserves the Popp measure (s. \cite{Bariz}) and hence, the intrinsic sublaplacians $\Delta_{\textup{sub}}^{T^\prime}$ and 
$\Delta_{\textup{sub}}^{T}$ are unitary equivalent. The last fact can be also directly seen from the representation of the intrinsic sublaplacian in \cite[Corollary 2]{Bariz} and 
the representation of $B_T$ in (\ref{Matrix_B_T}).  In particular, both sublaplacians have the same spectrum, i.e.  the spectrum of the trivializable subriemannian structure on $\mathbb{S}^7$ does not depend on the specific choice of the anti-commuting skew-symmetric matrices $A_j$.
\end{rem}
In the following we use the nilpotent approximation to compute the first heat invariant for the trivializable subriemannian structure endowed with the Popp measure. For this, let $z\in\mathbb{S}^7$ be fixed.
Since $X_1,\cdots,X_7$ is an adapted frame for $\mathbb{S}^7_T$ at $z$, the inverse of the local diffeomorphism
$$\phi^{-1}:(u_1,\cdots,u_7)\longmapsto \exp{(u_1X_1+\cdots+u_7X_7)}(z)$$
defines a system of local adapted coordinates at $z$. Because the adapted frame is a frame of linear vector fields, i.e.
$$X_i(z)=A_iz\hspace{2ex} \text{\it  for }\hspace{2ex} i=1,\cdots,7\hspace{2ex} \text{\it  and }\hspace{2ex} z\in\mathbb{S}^7,$$
the integral curve $\gamma(t)$ of the vector field 
$u_1X_1+\cdots+u_7X_7$ with $u=(u_1,\cdots,u_7)\in\mathbb{R}^7$ and starting at $z$ can be explicitly calculated as:
$$\gamma(t)=\cos{(\|u\|t)}z+\frac{\sin{(\|u\|t)}}{\|u\|}A_uz,$$
where 
$$A_u:=\sum_{i=1}^{7}u_iA_i\hspace{3ex} \text{\it  and }\hspace{3ex} \|u\|=\sqrt{u_1^2+\cdots+u_7^2}.$$
Hence $\phi^{-1}$ is given by:
$$\phi^{-1}(u)=\cos{(\|u\|)}z+\frac{\sin{\|u\|}}{\|u\|}A_uz\hspace{3ex} \text{\it  for }\hspace{3ex} u\in\mathbb{R}^7.$$
We recall that by the  anti-commutation relations (\ref{antic}) of the matrices $A_1,\cdots,A_7$, the matrix  $A_u$ fulfills the identity:
$$A^2_u=-\|u\|^2\text{Id}\hspace{3ex} \text{\it  for all } \hspace{3ex} u\in\mathbb{R}^7.$$
Now, let $w\in\mathbb{S}^7\backslash\{-z\}$ and let us consider the following equation in $u\in B(0,\pi):=\{u\in\mathbb{R}^7:\|u\|<\pi\}$:
\begin{equation}\label{locdif}
w=\cos{(\|u\|)}z+\frac{\sin{\|u\|}}{\|u\|}A_uz.
\end{equation}
Again by using the relations (\ref{antic}) we can write:
$$\langle w,z\rangle=\cos{\|u\|}\hspace{3ex} \text{\it  and }\hspace{3ex} \langle w,A_iz\rangle=\frac{\sin\|u\|}{\|u\|}u_i,$$
for $i=1,\cdots,7$. Hence the equation (\ref{locdif}) has the unique solution $u\in B(0,\pi)$ given by:
\begin{equation}\label{sol}
u_i=\frac{\arccos{\langle w,z\rangle}}{\sqrt{1-\langle w,z\rangle^2}}\langle w,A_iz\rangle \hspace{2ex} \text{\it  for }\hspace{2ex} i=1,\cdots,7.
\end{equation}
We summarize the above calculations in:
\begin{lem}
Canonical coordinates of the first kind at $z\in\mathbb{S}^7$ are given by
\begin{align*}
\phi:\mathbb{S}^7\backslash\{-z\}&\longrightarrow B(0,\pi)\\
w&\longmapsto \phi(w)=u,\\
\end{align*}
and $\phi(z)=0$, where $u$ is given by (\ref{sol}).
\end{lem}
Next, we compute the  expansion of the horizontal vector fields $X_1,\cdots,X_4$ near $0$ in the chart $\phi$.
Let us define the following smooth functions on $[0,\pi[$:
$$F(u):=\frac{1}{\|u\|^2}-\frac{\cot{\|u\|}}{\|u\|}\hspace{2ex} \text{\it  and }\hspace{2ex} G(u):=\|u\|\cot{\|u\|},$$
with $F(0):=\frac{1}{3}$ and $G(0):=1$.
\vspace{1mm} \\
Then a straightforward computation shows that the pushforwards of the horizontal vector fields $X_1,\cdots,X_4$ by $\phi$ are given on $B(0,\pi)$ by:
$$\left(X_i\right)_\ast=\sum_{j=1}^{7}a_{ij}\frac{\partial}{\partial u_j},$$
where the functions $a_{ij}$ with $b_{ik}^k  $ in (\ref{Definition_b_i_j_l}) are defined by:
\begin{equation}\label{ani}
a_{ij}(u):=G(u)\delta_{ij}+F(u)u_iu_j+\frac{1}{2}\sum_{k=1}^{7}b_{ij}^k(z)u_k.
\end{equation}
For $\epsilon>0$  small, consider the anisotropic expansion of $X_i$ around $0$:
$$X_i^\epsilon:=\epsilon\delta_\epsilon^\ast (X_i)_\ast\simeq X^{(-1)}+\epsilon X_i^{(0)}+\epsilon^2 X_i^{(1)}+\cdots,$$
where $X_i^{(l)}$ is the homogeneous part of $X_i$ of order $l$.

\begin{lem}\label{aniso}
For $i=1,\cdots,4$, it holds:
\begin{align*}
X^{(-1)}_i&=\frac{\partial}{\partial u_i}+\frac{1}{2}\sum_{j=5}^{7}\sum_{k=1}^{4}b_{ij}^ku_k\frac{\partial}{\partial u_j},\\
X_i^{(0)}&=\frac{1}{2}\sum_{j=1}^{4}\sum_{k=1}^{4}b_{ij}^ku_k\frac{\partial}{\partial u_j}+\frac{1}{2}\sum_{j=5}^{7}\sum_{k=5}^{7}b_{ij}^ku_k\frac{\partial}{\partial u_j},\\
X_i^{(1)}&=\frac{1}{2}\sum_{j=1}^{4}\sum_{k=5}^{7}b_{ij}^ku_k\frac{\partial}{\partial u_j}+\frac{1}{3}\sum_{j=1}^{7}u_iu_j\frac{\partial}{\partial u_j}-\frac{1}{3}\sum_{k=1}^{4}u_k^2\frac{\partial}{\partial u_i}.
\end{align*}
Furthermore, for $l\geq 2$:
$$X_i^{(l)}=G^{(l+1)}(u)\frac{\partial}{\partial u_i}+\sum_{j=1}^{7}F^{(l-1)}(u)u_iu_j\frac{\partial}{\partial u_j},$$
where  $F^{(l)}(u)$ (resp. $G^{(l)}(u)$) denotes the homogeneous part of weight $l$ in the anisotropic expansion of $F$ (resp. $G$).
\end{lem}
\begin{proof}
According to (\ref{ani}) we only need to compute the  expansion of $a_{ij}$ near $0$. We recall that the function $u\longmapsto u_i$ for $i=1,\cdots,4$ (resp. $i=5,\cdots,7$) has order $1$ (resp. $2$). Also the vector field $\frac{\partial}{\partial u_i}$ for $i=1,\cdots,4$ (resp. $i=5,\cdots,7$) has order $-1$ (resp. $-2$). The third term of (\ref{ani}):
$$\frac{1}{2}\sum_{k=1}^{7}b_{ij}^k(z)u_k=\frac{1}{2}\underbrace{\sum_{k=1}^{4}b_{ij}^k(z)u_k}_{\text{order }1}+\frac{1}{2}\underbrace{\sum_{k=5}^{7}b_{ij}^k(z)u_k}_{\text{order }2}$$ 
give us only homogeneous terms of order less than $2$. Furthermore, a straightforward calculation shows that for $\epsilon\rightarrow 0$:
\begin{align*}
F(\delta_\epsilon(u))&\simeq \frac{1}{3}+\sum_{l\geq 1}F^{(l)}(u)\epsilon^l\\
G(\delta_\epsilon(u))&\simeq 1-\frac{1}{3}\sum_{j=1}^{4}u_j^2\epsilon^2+\sum_{l\geq 3}G^{(l)}(u)\epsilon^l.
\end{align*}
Here $F^{(l)}(u)$ and $G^{(l)}(u)$ are homogeneous polynomials in $u$ of order $l$. By arranging homogeneous terms in the expression $(\ref{ani})$ and writing
$$X_i^{(l)}=\sum_{j=1}^{4}a_{ij}^{(l+1)}(u)\frac{\partial}{\partial u_j}+\sum_{j=5}^{7}a_{ij}^{(l+2)}(u)\frac{\partial}{\partial u_j},$$
where $a_{ij}^{(l)}$ denotes the homogeneous term of $a_{ij}$ of order $l$, we obtain the result.
\end{proof}
Note that Lemma \ref{aniso}  not only holds for the trivializable subriemannian structure defined by the  specific matrices $A_1,\cdots,A_7$ from Lemma \ref{Lemma_complement_to_anti-commuting_matrices}, but also for arbitrary skew-symmetric matrices with relations (\ref{antic}).
\vspace{1mm}\par 
Remark that only the first three homogeneous terms in the anisotropic expansion of $X_i$ encode the geometric data 
 $(b_{ij}^k)$ of our subriemannian manifold $\mathbb{S}^7_T$. The remaining homogeneous terms are completely given by the functions $F$ and $G$, which are independent of the chosen matrices $A_1,\cdots,A_7$.\ \\

The tangent group of $\mathbb{S}^7_T$ at $z$ is isomorphic to  the unique simply connected nilpotent Lie group $\widetilde{\mathbb{G}}(z)$ corresponding to the Lie algebra generated by the vector fields:
$$X_1^{(-1)},\cdots, X_4^{(-1)}.$$
By definition, the nilpotentization of the Popp measure at $z$ is the Haar measure $\hat{\mathcal{P}}_T^z$ on $\widetilde{\mathbb{G}}(z)\simeq \mathbb{R}^7$ given in global exponential coordinates $u_1,\cdots,u_7$ by:
$$\hat{\mathcal{P}}_T^z=g(z)du_1\wedge\cdots\wedge du_7.$$
Here $g(z)$ denotes the density appearing  in Lemma \ref{popp}.
 In order to compute the first heat invariant  $c_0$ we need to derive the heat kernel $p^{\widetilde{\mathbb{G}}(z)}$  of the sublaplacian
$$\hat{\Delta}_{\textup{sub}}^z:=\sum_{i=1}^{4}\left(X_i^{(-1)}\right)^2$$
on $\widetilde{\mathbb{G}}(z)\simeq\mathbb{R}^7$ with respect to the Haar measure $\hat{\mathcal{P}}_T^z$. This explicitly is obtained by the 
{\it Beals-Gaveau-Greiner formula} for the sublaplacian  on  general  step two nilpotent Lie groups in \cite{BGG,CCFI}, which we recall next.
For  $\alpha,\beta\in\widetilde{\mathbb{G}}(z)$ it holds:
\begin{equation}\label{beals}
{K_t^{\widetilde{\mathbb{G}}(z)}(\alpha,\beta)}=\frac{1}{(2\pi t)^{5}}\int_{\mathbb{R}^3}e^{-\frac{\varphi(\tau,\,\alpha^{-1}\: * \: \beta)}{t}}\,W(\tau)\,\frac{d\tau}{g(z)},
\end{equation}
where the  {\it action function} $\varphi=\varphi(\tau,\alpha)\in C^{\infty}(\mathbb{R}^{3}\times
\widetilde{\mathbb{G}}(z))$ and the  {\it volume element} $W(\tau)\in C^{\infty}(\mathbb{R}^3)$ are given as follows:
put $\alpha=(a,b)\in \mathbb{R}^4\times\mathbb{R}^3$, then
\begin{align*}
&\varphi(\tau,g)=\varphi(\tau,a,b)=\sqrt{-1}\langle \tau,\,b \rangle 
+\frac{1}{2}\Big{\langle}\sqrt{-1}J_{\tau/2}\coth\bigr(\sqrt{-1}J_{\tau/2}\bigr)
\cdot a,a\Big{\rangle},\\
&W(\tau)=
\left\{\det\frac{\sqrt{-1}J_{\tau/2}}{\sinh \sqrt{-1}J_{\tau/2}}\right\}^{1/2},
\end{align*}
where $\langle b,b^{\prime}\rangle=\sum\limits_{k=1}^3 b_k b_k^{\prime}$ denotes the Euclidean inner product on $\mathbb{R}^3$. 

Next, we compute the eigenvalues of the representation maps $J_Z$, for $Z\in \mathcal{V}_z\simeq\mathbb{R}^3.$
\begin{lem}\label{eig}
Let $z=(x,y)\in\mathbb{S}^7$ and $Z\in\mathcal{V}_z$. Then the eigenvalues of $J_Z$ are 
$$\pm2i(\|x\|^2\pm\|y\|^2)\|Z\|.$$
\end{lem}
\begin{proof}
 According to (\ref{max}) the characteristic polynomial $P(\lambda)$ of $J_Z$ is given by: 
$$P(\lambda)=\lambda^4+8(1-2\|x\|^2\|y\|^2)\|Z\|^2\lambda^2+16(1-4\|x\|^2\|y\|^2)\|Z\|^4.$$
Hence, a straightforward calculation shows that the roots of $P(\lambda)$ are exactly
$$\pm2i(\|x\|^2\pm\|y\|^2)\|Z\|.$$
\end{proof}

\begin{thm}
The first heat invariant $c_0^T$ of the trivializable subriemannian structure on $\mathbb{S}^7$ is given by
$$c_0^T(z)=\frac{1}{(2\pi)^5g(z)}\int_{\mathbb{R}^3}\frac{\|\tau\|}{\sinh{\|\tau\|}}\cdot\frac{(\|x\|^2-\|y\|^2)\|\tau\|}{\sinh{(\|x\|^2-\|y\|^2)\|\tau\|}}d\tau$$
for $z=(x,y)\in\mathbb{S}^7$.
\end{thm}
\begin{proof}
let $z=(x,y)\in\mathbb{S}^7$ and $Z\in\mathcal{V}_z$.  By Lemma \ref{eig}, the eigenvalues of the skew-symmetric operator $J_Z$  are $\pm2i(\|x\|^2\pm\|y\|^2)\|Z\|$. We assume that $z$ fulfills:
 $$\|x\|\neq\|y\| \hspace{2ex} \text{\it  and }\hspace{2ex} x \neq 0.$$
Such points form a dense subset in $\mathbb{S}^7$ and therefore, due to the {smoothness  of the local assignment $z\longmapsto c_0^T(z)$ (see \cite{Verd})} we only need to compute $c_0^T(z)$ at such points. The  advantage of considering such points is that the eigenvalues of the map $J_Z$ for all $Z\in\mathcal{V}_z$, are simple. Hence the expression of the function $W(\tau)$ take the form: 
$$\det{\left(\frac{iJ_\tau/2}{\sinh{(iJ_\tau/2)}}\right)}=\left(\frac{\|\tau\|}{\sinh{\|\tau\|}}\right)^2\left(\frac{(\|x\|^2-\|y\|^2)\|\tau\|}{\sinh{((\|x\|^2-\|y\|^2)\|\tau\|)}}\right)^2.$$
Hence, by (\ref{firstheat}) and (\ref{beals}), we can write:
\begin{align*}
c_0^T(z)&=\frac{1}{(2\pi)^5}\int_{\mathbb{R}^3} \sqrt{\det{\left(\frac{iJ_\tau/2}{\sinh{iJ_\tau/2}}\right)}}\frac{d\tau}{g(z)}\\
&=\frac{1}{(2\pi)^5g(z)}\int_{\mathbb{R}^3}\frac{\|\tau\|}{\sinh{\|\tau\|}}\cdot\frac{(\|x\|^2-\|y\|^2)\|\tau\|}{\sinh{(\|x\|^2-\|y\|^2)\|\tau\|}}d\tau.
\end{align*}
\end{proof}
\begin{rem}
At points $z=(x,y)\in\mathbb{S}^7$ with $x=0$ or $y=0$, a straightforward computation  using the representation (\ref{max}) shows that the maps $J_{X_5},J_{X_6}$ and $J_{X_7}$ fulfill the quaternionic relations  and hence the tangent groups of the subriemannian manifolds $\mathbb{S}^7_T$ and $\mathbb{S}^7_Q$ are isometric. 
Furthermore, to compute the first heat invariant at $z$ we only need to know the subriemannian structure at this point and hence it follows that at these points 
 the first heat invariants coincide:   
$$c_0^Q(z)=c_0^T(z).$$
Also it is not hard to see that the infimum of $c_0^T(z)$ over $\mathbb{S}^7$ is attained at these points and therefore we can write

\begin{equation}\label{inf}
\inf\{c_0^T(z):z\in\mathbb{S}^7\}=\hat{c}_0^Q.
\end{equation}
Here $\hat{c}_0^Q$ denotes the value of the constant function $z\longmapsto c_0^Q(z)$ which will be calculated explicitly below.
\vspace{1mm} \par 
We remark that the remaining heat invariants $c_1,c_2,\cdots$ might not be equal at these special points.  In fact, in order to compute these numbers we have to take into account the local behavior of the corresponding subriemannian structures at such points.
\end{rem}
As a corollary we prove now that the subriemannian manifolds $\mathbb{S}^7_T$ and $\mathbb{S}^7_Q$ are not {\it isospectral} with respect to the 
intrinsic sublaplacians:
\begin{cor}\label{noniso}
Let $\mathbb{S}^7_T$ and $\mathbb{S}^7_Q$ be considered with the induced Popp measures. Then the intrinsic sublaplacians 
$\Delta_{\textup{sub}}^T$ and $\Delta_{\textup{sub}}^Q$ are not isospectral.
\end{cor}
\begin{proof}
By considering the subriemannian manifold $\mathbb{S}^7_Q$ as a quaternionic contact manifold and using the quaternionic relations of the 
almost complex structures $I_{\bf l}$ for ${\bf l}\in\{{\bf i},{\bf j},{\bf k}\}$, we see that the Popp measure is given by, (see Lemma \ref{Popp_volume_QSS}):
$$\mathcal{P}_Q(z)=\frac{1}{(16)^{3/2}}d\sigma(z).$$
Furthermore, the nilpotent approximation of $\mathbb{S}^7_Q$ at $z\in\mathbb{S}^7$ is isomorphic to the standard quaternionic Heisenberg group. Hence the first heat invariant of $\mathbb{S}^7_Q$ is given by
$$c_0^Q(z)=\frac{16^{3/2}}{(2\pi)^5}\int_{\mathbb{R}^3}\left(\frac{\|\tau\|}{\sinh{\|\tau\|}}\right)^2d\tau\text{ for }z\in\mathbb{S}^7.$$
We set 
\begin{multline}\label{isos}
c_0^T-c_0^Q:=\\\frac{1}{(2\pi)^5}\int_{\mathbb{S}^7}\int_{\mathbb{R}^3}\frac{\|\tau\|}{\sinh{\|\tau\|}}\left(\frac{(\|x\|^2-\|y\|^2)\|\tau\|}{\sinh{((\|x\|^2-\|y\|^2)\|\tau\|)}}-\frac{\|\tau\|}{\sinh{\|\tau\|}}\right)d\tau d\sigma(z).
\end{multline}
 Note that the function $u\longmapsto u/\sinh{(u)}$ is  even, smooth and monotone decreasing on the interval $[0,\infty[$. This shows that the integrand in (\ref{isos}) is a  non-negative function on $\mathbb{S}^7\times\mathbb{R}^3$ and non-vanishing on an open dense subset. Therefore $c_0^T>c_0^Q$ and 
 the subriemannian manifolds $\mathbb{S}^7_Q$ and $\mathbb{S}^7_T$ cannot be isospectral.
\end{proof}

\section{Sublaplacian induced by the standard measure}
\label{Chapter_9}
If we consider the subriemannian manifold $\mathbb{S}^7_T$ endowed with the standard volume $d\sigma$, then the corresponding sublaplacian 
$\widetilde{\Delta}_{\textup{sub}}^{T}$ will be a sum of squares:
\begin{equation}\label{sum_of_squares_operator}
\widetilde{\Delta}_{\textup{sub}}^{T}= -\sum_{i=1}^{4}X_i^2.
\end{equation}
 Here $X_i=X(A_i)$ for $i=1, \ldots, 4$ with $A_j$ defined in (\ref{matrices_A_j}) and Lemma \ref{Lemma_complement_to_anti-commuting_matrices} 
denotes the system of linear vector fields generating the distribution $\mathcal{H}_T$ of $\mathbb{S}^7_T$.
According to \cite{Hoe67} the operator (\ref{sum_of_squares_operator}) is subelliptic, positive and with discrete spectrum consisting of eigenvalues. We recall that a part of 
this spectrum has been determined in \cite{BFI}. Moreover, Corollary 5.4 of  \cite{BFI} implies that a different choice of the generating anti-commuting 
skew-symmetric matrices $A_j$ leads to a sublaplacian which is unitary equivalent to (\ref{sum_of_squares_operator}) and therefore has the same spectrum. Hence, when 
studying the spectrum of the trivializable subriemannian structure, we can restrict ourselves to a specific choice the generators of a Clifford algebra 
(s. Remarks \ref{Remark_different_choices_of_generators} and \ref{Remark_Isospectrality_of_equivalent_triv_structures}). 
\vspace{1mm}\par 
In this section a relation between the spectrum of $\widetilde{\Delta}_{\textup{sub}}^{T}$ and the spectrum of the sublaplacian 
\begin{equation*}
\Delta_{\textup{sub}}^{Q}=\Delta_{\mathbb{S}^7} + X(A_6)^2+X(A_7)^2+X(A_6A_7)^2 
\end{equation*}
induced by the quaternionic Hopf fibration (s. \cite{Baud_Wang}) will be shown. Here 
\begin{equation*}
 \Delta_{\mathbb{S}^7}= - \sum_{j=1}^7 X(A_j)^2
\end{equation*} 
denotes the Laplace-Beltrami operator on $\mathbb{S}^7$ with respect to the standard metric. Via the inclusion $\mathbb{S}^3 \subset (\mathbb{H}, *)$ and for 
$\ell\in \{ {\bf i}, {\bf j}, {\bf k} \}$ consider the 
vector fields: 
\begin{equation*} 
W_{\ell}f(z)= {\frac{\rm d}{\rm dt}_|}_{t=0}f\big{(} z * e^{t{\ell}} \big{)}  \hspace{3ex}\mbox{\it where} \hspace{3ex} f \in C^{\infty}(\mathbb{S}^3). 
\end{equation*}
\par 
By the same formula $W_{\ell}$ can be interpreted as a (linear) vector field on $\mathbb{R}^4\cong\mathbb{H}$. A direct calculation using the decomposition 
$(x,y) \in \mathbb{S}^7 \subset \mathbb{R}^4\times \mathbb{R}^4 \subset \mathbb{H}^2$ and the form of the matrices in (\ref{matrices_A_j}) shows: 
\begin{align}\label{representation_tensor_SL_1}
\widetilde{\Delta}_{\textup{sub}}^{T}&=\Delta_{\mathbb{S}^7} - \Delta_{\mathbb{S}^3} \otimes I - I \otimes \Delta_{\mathbb{S}^3} - 2B\\
\Delta_{\textup{sub}}^Q&=\Delta_{\mathbb{S}^7} - \Delta_{\mathbb{S}^3} \otimes I - I \otimes \Delta_{\mathbb{S}^3} + 2B. \label{representation_tensor_SL_2}
\end{align}
Here $\Delta_{\mathbb{S}^3}=-\sum_{\ell \in \{{\bf i}, {\bf j}, {\bf k}\}} W_{\ell}^2$ denotes the Laplace-Beltrami operator on $\mathbb{S}^3$ with respect to the 
standard metric and 
\begin{equation}\label{operator_B}
B:= \sum_{\ell \in \{ {\bf i}, {\bf j}, {\bf k}\}} W_{\ell} \otimes W_{\ell}. 
\end{equation}
The tensor product notation $A \otimes C$ means that an operator $A$ acts with respect to the variable $x$ and $C$ with respect to $y$. 
Note that $B$ in (\ref{operator_B}) vanishes on smooth functions {$f(x,y)= \tilde{f}(x)$} and $g(x,y)= \tilde{g}(y)$ which only depend on $x$ and 
$y$ of $\mathbb{R}^4$, respectively. Therefore,  $\widetilde{\Delta}_{\textup{sub}}^{T}$ and $\Delta_{\textup{sub}}^Q$ 
act in the same way on functions $g$ and $f$ of the  above type. 
\vspace{1ex} \par
With the notation $\omega=(\omega_1, \omega_2)\in \mathbb{R}^4 \times \mathbb{R}^4$ we write $K_t^{\textup{Q}}(\omega,\textup{np})= \widetilde{k}_t^{\textup{Q}}(\omega_1)$ for the heat kernel of 
 $\Delta_{\textup{sub}}^{\textup{Q}}$ at the north pole $\textup{np}=(1,0, \ldots, 0)\in \mathbb{S}^7 \subset \mathbb{R}^8$. The function  $ \widetilde{k}_t^{\textup{Q}}$ 
 has been calculated in \cite{Baud_Wang} and only depends on $\omega_1$. It  follows from the previous remark that: 
\begin{equation}\label{GL_differential_equation_Delta_4_Delta_Hopf}
\widetilde{\Delta}_{\textup{sub}}^{T}K_t^{\textup{Q}}(\cdot,\textup{np})= \Delta_{\textup{sub}}^{\textup{Q}}K_t^{\textup{Q}}(\cdot,\textup{np})=- \frac{\rm d}{\rm dt} K_t^{\textup{Q}}(\cdot,\textup{np}). 
\end{equation}
\par 
Choose an orthonormal system $[\phi_{\ell}\: : \: \ell \in \mathbb{N}]$ of $L^2(\mathbb{S}^7)=L^2(\mathbb{S}^7, \sigma)$ consisting of smooth eigenfunctions of $\widetilde{\Delta}_{\textup{sub}}^T$ 
with corresponding eigenvalues $\lambda_{\ell}\geq 0$. We obtain an expansion of the heat kernel: 
\begin{equation*}
K_t^{\textup{Q}}(\omega,\textup{np})=\sum_{\ell=1}^{\infty} c_{\ell}(t) \phi_{\ell}(\omega)=\sum_{\ell=1}^{\infty} c_{\ell}(t) \phi_{\ell}(\omega_1,0), 
\end{equation*}
which converges in $C^{\infty}(\mathbb{S}^7)$. From (\ref{GL_differential_equation_Delta_4_Delta_Hopf}) one concludes that 
$c_{\ell}^{\prime}(t)+ \lambda_{\ell} \cdot c_{\ell}(t)=0$ for each $\ell \in \mathbb{N}$. Hence there are constants $\gamma_{\ell}$ such that: 
\begin{equation*}
c_{\ell}(t)= \gamma_{\ell} e^{-\lambda_{\ell}t} \hspace{4ex} \mbox{\it where} \hspace{4ex} t >0. 
\end{equation*}
Moreover, for all $\ell \in \mathbb{N}$: 
\begin{equation*}
\phi_{\ell}(\textup{np})= \lim_{t\downarrow 0} \int_{\mathbb{S}^7} \phi_{\ell}(\omega) K_t^{\textup{Q}}(\omega, \textup{np})d\sigma(\omega) 
= \lim_{t \downarrow 0} c_{\ell}(t)= \gamma_{\ell}. 
\end{equation*}
Let $K_t^{T}(\omega, \textup{np})$ denote the heat kernel of $\widetilde{\Delta}_{\textup{sub}}^T$. From our calculation we conclude: 
\begin{equation}\label{K_H_=_K_4}
K_t^{\textup{Q}}(\omega, \textup{np})= \sum_{\ell=1}^{\infty}  e^{- \lambda_{\ell} t} \phi_{\ell}(\omega) \phi_{\ell}(\textup{np}).
\end{equation}
\par 
On the other hand, we can choose an orthonormal basis $[\psi_{\ell} \: : \: \ell \in \mathbb{N} ]$ of $L^2(\mathbb{S}^7)$ consisting of eigenfunctions of 
$\Delta_{\textup{sub}}^{\textup{Q}}$ with corresponding eigenvalue sequence $(\mu_{\ell})_{\ell \in \mathbb{N}}$. We write: 
\begin{equation}\label{K_H_=_K_4_second}
K_t^{\textup{Q}}(\omega, \textup{np})=\sum_{\ell=1}^{\infty} e^{- \mu_{\ell}t} \psi_{\ell}(\omega) \psi_{\ell}(\textup{np})= \sum_{\ell=1}^{\infty} e^{- \widetilde{\mu}_{\ell} t} \Psi_{\ell} (\omega). 
\end{equation}
On the right hand side we have used the definition: 
\begin{equation*}
\Psi_{\ell} (\omega):=\sum_{j \: \textup{s.t.} \atop \mu_j = \widetilde{\mu}_{\ell}} \psi_j(\omega) \psi_j(\textup{np}), 
\end{equation*}
where $0 \leq  \widetilde{\mu}_1 < \widetilde{\mu}_2 < \widetilde{\mu}_3 \ldots $ denotes the sequence of distinct eigenvalues of 
$\Delta_{\textup{sub}}^{\textup{Q}}$ in increasing order. 
We write $m(\mu)$ for the multiplicity of an eigenvalue $\mu$ of $\Delta_{\textup{sub}}^{\textup{Q}}$.
\begin{lem}\label{Lemma_multiplicities}
For $\ell \in \mathbb{N}$ the sum $\sum_{\mu_j= \widetilde{\mu}_{\ell}} |\psi_{j}(x)|^2\equiv \| \Psi_{\ell}\|_{L^2(\mathbb{S}^7)}^2$ is constant on $\mathbb{S}^7$ and 
$$m(\mu_{\ell})= \textup{vol}(\mathbb{S}^7) \| \Psi_{\ell}\|_{L^2(\mathbb{S}^7)}^2 \ne 0.$$  
\end{lem}
\begin{proof}
Since $\{\psi_{\ell}\}_{\ell}$ is an orthonormal basis of $L^2(\mathbb{S}^7)$ we have: 
\begin{equation*}
\| \Psi_{\ell}\|_{L^2(\mathbb{S}^7)}^2= \sum_{j \: \textup{s.t.} \atop \mu_{j}= \widetilde{\mu}_{\ell}} |\psi_{j}(\textup{np})|^2. 
\end{equation*}
Consider the subriemannian isometry group $\mathcal{I}(\mathbb{S}_Q^7)$. Recall that ${\bf Sp}(2)\subset  \mathcal{I}(\mathbb{S}_Q^7)$ and ${\bf Sp}(2)$ acts transitively on $\mathbb{S}^7$ (see the proof of Lemma \ref{Popp_volume_QSS}). For all $g \in {\bf Sp}(2)$ we define the unitary operator $V_g$ on 
$L^2(\mathbb{S}^7)$ by composition, i.e. $V_gf:= f \circ g$ for all $f \in L^2(\mathbb{S}^7)$. Note that 
\begin{equation*}
\big{[} \Delta^{\textup{Q}}_{\textup{sub}}, V_g\big{]}=0 
\end{equation*}
and put $\psi^g_{\ell}:= V_g\psi_{\ell}= \psi_{\ell} \circ g$. Then $\{\psi^g_{\ell}\}_{\ell}$ defines an orthonormal basis of $L^2(\mathbb{S}^7)$ consisting of 
eigenfunctions of $\Delta^{\textup{Q}}_{\textup{sub}}$ corresponding to the sequence $(\mu_{\ell})_{\ell}$ of eigenvalues, as well, and the heat kernel expansion 
of $\Delta_{\textup{sub}}^{\textup{Q}}$ can be rewritten as: 
\begin{equation*}
K_t^{\textup{Q}}(\omega, \textup{np})= \sum_{\ell=1}^{\infty} e^{- \mu_{\ell}t} \psi_{\ell} \circ g(\omega) \cdot \psi_{\ell} \circ g(\textup{np}), \hspace{2ex} \mbox{\it where} \hspace{2ex} \omega \in \mathbb{S}^7. 
\end{equation*}
It follows for all $g \in H$: 
\begin{equation*}
\| \Psi_{\ell}\|_{L^2(\mathbb{S}^7)}^2= \sum_{j \: \textup{s.t.} \atop \mu_j= \widetilde{\mu}_{\ell}} |\psi_{j}\circ g(\textup{np})|^2. 
\end{equation*}
Since ${\bf Sp}(2)$ acts transitively on $\mathbb{S}^7$ we conclude that the finite sum below is constant on $\mathbb{S}^7$ with value: 
\begin{equation}\label{equality_of_functions_S_7}
\sum_{j \: \textup{s.t.}\atop \mu_j=\widetilde{\mu}_{\ell}}|\psi_j(x)|^2 \equiv \| \Psi_{\ell}\|_{L^2(\mathbb{S}^7)}^2, \hspace{8ex} (x \in \mathbb{S}^7). 
\end{equation}
Hence: 
\begin{equation*}
m(\widetilde{\mu}_{\ell})
=\# \big{\{}j \: : \mu_j= \widetilde{\mu}_{\ell} \big{\}}= \int_{\mathbb{S}^7} \sum_{j \: \textup{s.t.} \atop \mu_{j}=\widetilde{\mu}_{\ell}}|\psi_{j}(x)|^2d \sigma(x)= 
\textup{vol}(\mathbb{S}^7)\| \Psi_{\ell}\|_{L^2(\mathbb{S}^7)}^2. 
\end{equation*}
This proves the assertion. 
\end{proof}
Lemma \ref{Lemma_multiplicities} implies that in each eigenspace of $\Delta_{\textup{sub}}^{\textup{Q}}$ there is an element $\psi$ such that $\psi(\textup{np}) \ne 0$. 
As usual let $\sigma(A)$ denote the spectrum of an operator $A$ and put: 
\begin{equation*}
\Lambda:= \Big{\{} \lambda \in \sigma \big{(}\widetilde{\Delta}_{\textup{sub}}^{\textup{T}}\big{)} \: : \: \exists \: \phi \in  
\textup{ker}\big{(}\widetilde{\Delta}_{\textup{sub}}^{\textup{T}}- \lambda\big{)}\: \textup{\it such that } \: \phi(\textup{np}) \ne 0\Big{\}}. 
\end{equation*}
Consider the following subset of distinct eigenvalues: 
\begin{equation*}
\Lambda:=\big{\{} \widetilde{\lambda}_{\ell} \in \Lambda \: : \: \widetilde{\lambda}_1 < \widetilde{\lambda}_2 < \ldots\big{\}} \subset  
\sigma \big{(}\widetilde{\Delta}_{\textup{sub}}^{\textup{T}}\big{)}. 
\end{equation*}
From (\ref{K_H_=_K_4}) and (\ref{K_H_=_K_4_second}) we have for all $t>0$: 
\begin{equation} \label{GL_equality_of_heat_kernels}
\sum_{\ell=1}^{\infty} e^{-\widetilde{\lambda}_{\ell} t} \Phi_{\ell}(\omega) = \sum_{\ell=1}^{\infty} e^{- \widetilde{\mu}_{\ell} t} \Psi_{\ell} (\omega), 
\end{equation}
where for each $\widetilde{\lambda}_{\ell} \in \Lambda $: 
\begin{equation*}
\Phi_{\ell}(\omega) := \sum_{j \: \textup{s.t.} \atop \lambda_j= \widetilde{\lambda}_{\ell}} \phi_j(\omega) \phi_j (\textup{np}).
\end{equation*}
Note that $\Phi_{\ell}(\textup{np}) \ne 0$ by definition of $\Lambda$. 
\begin{thm}\label{Lemma_inclusion_of_spectra_sublaplace}
We have the inclusion of spectra $\Lambda= \sigma(\Delta^{\textup{Q}}_{\textup{sub}}) \subset \sigma (\widetilde{\Delta}_{\textup{sub}}^{\textup{T}})$. 
\end{thm}
\begin{proof}
Assume that $\widetilde{\mu}_1 \ne \widetilde{\lambda_1}$. Without loss of generality assume that 
$\widetilde{\lambda}_1 < \widetilde{\mu}_1$. Then 
\begin{equation*}
0 \ne \Phi_1(\textup{np})= \sum_{\ell=1}^{\infty}e^{-(\widetilde{\mu}_{\ell}- \widetilde{\lambda}_1)t}\Psi_{\ell}(\textup{np})-
\sum_{\ell=2}^{\infty} e^{- (\widetilde{\lambda}_{\ell}- \widetilde{\lambda}_1)t} \Phi_{\ell}(\textup{np}). 
\end{equation*}
Since the right hand side tends to zero as $t \rightarrow \infty$ we obtain a contradiction.  Hence $\widetilde{\lambda}_1 = \widetilde{\mu}_1$ and 
\begin{equation*}
\Phi_1(\textup{np})= \Psi_1(\textup{np})=m(\mu_1). 
\end{equation*}
Therefore
\begin{equation*}
\sum_{\ell=2}^{\infty} e^{-\widetilde{\lambda}_{\ell} t} \Phi_{\ell}(\textup{np}) = \sum_{\ell=2}^{\infty} e^{- \widetilde{\mu}_{\ell} t} \Psi_{\ell} (\textup{np}), 
\end{equation*}
and proceeding inductivley in this way we obtain the result. 
\end{proof}
\begin{rem}
The spectrum $\sigma\big{(}\Delta^{\textup{Q}}_{\textup{sub}}\big{)}$ is known explicitly, see \cite{Baud_Wang}.  Moreover, the multiplicities of eigenvalues 
$\lambda\in \Lambda$ with respect to the operators $\Delta^{\textup{Q}}_{\textup{sub}}$ and $\widetilde{\Delta}_{\textup{sub}}^{\textup{T}}$  
may not coincide. The statement in Theorem \ref{Lemma_inclusion_of_spectra_sublaplace} generalizes results in \cite{BFI} where a (smaller) part of the spectrum 
$\sigma(\widetilde{\Delta}_{\textup{sub}}^{T})$ has been calculated. 
\end{rem} 

Using the nilpotent approximation as in Section \ref{small} we can show that the first heat invariant $\tilde{c}_0^{T}(z)$ of $\widetilde{\Delta}_{\textup{sub}}^{T}$ is exactly:
$$\tilde{c}_0^{T}(z)=\frac{1}{16^{3/2}(2\pi)^5}\int_{\mathbb{R}^3}\frac{\|\tau\|}{\sinh{\|\tau\|}}\cdot\frac{(\|x\|^2-\|y\|^2)\|\tau\|}{\sinh{(\|x\|^2-\|y\|^2)\|\tau\|}}d\tau$$
for $z=(x,y)\in\mathbb{S}^7$. Hence using the same arguments as in the proof of Corollary \ref{noniso}, it follows that the operators $\Delta_{\textup{sub}}^Q$ 
and $\widetilde{\Delta}_{\textup{sub}}^{T}$ are not isospectral, as well, i.e. the inclusion of spectra in Theorem \ref{Lemma_inclusion_of_spectra_sublaplace} is strict {or they have the same spectrum but the eigenvalues have different multiplicities. }
\vspace{1mm}
\par 
Since the sublaplacian $\widetilde{\Delta}_{\textup{sub}}^{T}$ is a {\it sum-of-squares operator}, using the vector fields $X_i^{(-1)},X_i^{(0)},X_i^{(1)}$ ($i=1,\cdots,4$) from Lemma \ref{noniso} and the expression (\ref{second_invariant})  we can obtain a formula for the second heat invariant  $\tilde{c}_1^{T}$ which shows how 
this quantity depends on the geometric data $(b_{ij}^k)$. But due to the small symmetry group  of the trivializable SR structure (it does not act transitively)  the 
calculation is complicated and we omit it here. 

\section{Open problems}
Finally, we mention some open problems  which have been left in the analysis of the trivializable subriemannian manifold $\mathbb{S}^7_T$.
\begin{enumerate}
\item What is the significance of the second heat invariant $c_1^{T}$ for the trivializable subriemannian structure on $\mathbb{S}^7$? We recall that in the framework of Riemannian geometry, the second heat invariant can be interpreted as {integrals of curvature tensors over the manifold}. Furthermore, for contact subriemannian structures on 3-dimensional manifolds an interpretation of the second heat invariant in terms of certain curvature terms has given by D. Barilari in \cite{Bari}. 
\item Derive an explicit formula for the heat kernel of the sublaplacian $\Delta_{\textup{sub}}^T$   on $\mathbb{S}^7_T$ and on $\mathbb{S}^{15}_T$ equipped with the rank eight trivializable subriemannian structure of step two in \cite{BFI}.  In case of the quaternionic contact structure such a
 formula is known and can be found in \cite{Baud_Wang}.
\item What is the dimension of the subriemannian isometry group $\mathcal{I}(\mathbb{S}^7_T)$?  
\item  As is known, the Carnot-Carath\'{e}odory distance on $\mathbb{S}_T^{7}$ appears in the exponent of the off-diagonal small time asymptotics of the subelliptic heat kernel of $\Delta^T_{\textup{sub}}$. Can one (at least locally) obtain formulas or estimates on $d$ via a heat kernel analysis? 
\end{enumerate}


 
\end{document}